\documentclass[a4paper]{amsart}

\usepackage{amssymb}
\usepackage{xcolor}
\usepackage{eucal}

\usepackage{filecontents}

\usepackage{amssymb}
\usepackage{amsmath}
\usepackage{amsthm}
\usepackage{amsaddr}
\usepackage{orcidlink}
\usepackage{enumitem}

\allowdisplaybreaks

\usepackage[style=numeric, url=false, doi=false, backend=bibtex]{biblatex}
\theoremstyle{definition}

\newtheorem{mydef}{Definition}[section]

\newtheorem{myque}[mydef]{Question}
\newtheorem*{myack}{Acknowledgements}

\theoremstyle{remark}

\newtheorem{myex}[mydef]{Example}

\theoremstyle{plain}

\newtheorem{mycol}[mydef]{Corollary}
\newtheorem{mysen}[mydef]{Theorem}
\newtheorem{mylem}[mydef]{Lemma}

\newtheorem{myfact}[mydef]{Fact}
\newtheorem*{myclaim}{Claim}

\newtheorem{mysenx}{Theorem}

\numberwithin{mydef}{section}

\DeclareMathOperator{\dom}{dom}

\DeclareMathOperator{\cf}{cf}

\DeclareMathOperator{\Add}{Add}

\DeclareMathOperator{\ISP}{ISP}
\DeclareMathOperator{\AP}{AP}
\DeclareMathOperator{\PFA}{PFA}
\DeclareMathOperator{\Coll}{Coll}
\DeclareMathOperator{\Odd}{Odd}

\DeclareMathOperator{\op}{op}
\DeclareMathOperator{\TP}{TP}

\DeclareMathOperator{\ITP}{ITP}

\DeclareMathOperator{\tcl}{tcl}

\newcommand{\dL}{\mathbb{L}}
\newcommand{\dM}{\mathbb{M}}
\newcommand{\dP}{\mathbb{P}}
\newcommand{\dQ}{\mathbb{Q}}
\newcommand{\dR}{\mathbb{R}}

\newcommand{\dT}{\mathbb{T}}

\newcommand{\uhr}{\upharpoonright}

\subjclass[2020]{03E05,03E35,03E55} %TODO

\title{Slender trees and the approximation property}
\author{Hannes Jakob}
\address{Universit\"at Freiburg\\Institut f\"ur Mathematik\\Ernst-Zermelo-Stra{\ss}e 1, 79104 Freiburg im Breisgau, Germany\\ORCiD:\orcidlink{https://orcid.org/0009-0006-4563-1369}}
\email{hannes.jakob@mathematik.uni-freiburg.de}

\date{\today}

\begin{document}
	
	%For IMPAN again
	
	\baselineskip=17pt
	\keywords{Slender trees, large cardinals, forcing, independence results} %TODO
	
	%End
	
	\begin{abstract}
		We obtain a relatively simple criterion for when a forcing has the ${<}\,\delta$-approximation property, generalizing a result of Unger. Afterwards we apply this criterion to construct variants of Mitchell Forcing in order to answer questions posed by Mohammadpour.
	\end{abstract}
	
	\maketitle
	
	\section{Introduction}
	
	In his PhD thesis, Christoph Weiss introduced a two-cardinal generalisation of a $\kappa$-tree called a $(\kappa,\lambda)$-list along with two ``thinness`` properties, namely \emph{thin} and \emph{slender}. The corresponding cardinal properties, stating that every ``not too wide`` list has an ineffable branch, behave very differently in practice. The \emph{ineffable tree property}, denoted $\ITP$, is consistent with e.g. the approachability property and forced by many orders which were originally conceived with the tree property in mind. On the other hand, the \emph{ineffable slender tree property}, denoted $\ISP$, forces that $\AP$ fails and needs different preservation theorems compared to the tree property.
	
	The proof structures for consistency proofs involving $\ISP$ and $\ITP$ are the same: Given a list in the extension $V[G]$, lift some ground-model embedding $V\to M$ to an embedding $V[G]\to M[H]$ inside $V[H]$, thus obtaining an ineffable branch in $V[H]$ and ``pull it back`` into $V[G]$ using the thinness assumption and properties of the extension $V[H]$ over $V[G]$. For thin lists, this is done using two branch lemmas stating that sufficiently closed forcings and forcings whose squares have small antichains do not add cofinal branches to thin lists. For $\delta$-slender lists however, we need that the pair $(V[G], V[H])$ has the so-called ${<}\,\delta$-approximation property. In practice, this property is much harder to establish and so many results which have been shown for $\TP$ and $\ITP$ have not yet been established for $\ISP$. E.g. it was shown by Cummings and Foreman in \cite{CummingsForemanTreeProperty} that the tree property can consistently hold at every $\aleph_{n+2}$, which was improved by Fontanella in \cite{FontanellaTPInterval} to show that $\ITP$ can consistently hold at every $\aleph_{n+2}$. On the other hand the only comparable result for $\ISP$ was obtained by Mohammadpour and Velickovic in \cite{MohamVelicGuessingModelsApproach} who showed using side condition forcing that consistently $\ISP(\omega_2)$ and $\ISP(\omega_3)$ can hold simultaneously.
	
	In this paper, we obtain two important tools to prove consistency results involving $\ISP$: We obtain a general criterion for when a forcing order preserves the statement $\ISP$ and, more importantly, prove the following statement concerning the approximation property, generalizing a result of Unger from \cite{UngerFragIndestructII}:
	
	\begin{mysenx}
		Let $(\dP\times\dQ,R)$ be an iteration-like partial order and $\delta$ a cardinal. Assume that the base ordering $(\dP,b(R))$ is square-$\delta$-cc. and the term ordering $(\dP\times\dQ,t(R))$ is strongly ${<}\,\delta$-distributive. Then $(\dP\times\dQ,R)$ has the ${<}\,\delta$-approximation property.
	\end{mysenx}
	
	We then demonstrate the robustness of the framework by constructing several new variants of Mitchell forcing in order to answer questions posed by Mohammadpour in a recent survey (see \cite{MohammadpourRoadCompactness}). Recall that a cardinal $\kappa$ is \emph{$\lambda$-ineffable} if $\kappa$ is inaccessible and $\ISP(\kappa,\lambda)$ holds.
	
	\begin{mysenx}
		Let $\tau<\mu<\kappa\leq\lambda$ be regular cardinals such that $\tau^{<\tau}=\tau$ and $\kappa$ is $\lambda$-ineffable. There exist forcing extensions satisfying the following (with $\kappa=\mu^+$):
		\begin{enumerate}
			\item $\ISP(\tau^+,\kappa,\lambda)$ holds, $\ISP(\tau,\kappa,\lambda)$ fails and $2^{\tau}$ is an arbitrarily large cardinal below $\lambda$.
			\item $\ISP(\kappa,\kappa,\lambda)$ holds, $\ISP(\mu,\kappa,\kappa)$ fails and $2^{\tau}\neq 2^{\mu}$.
		\end{enumerate}
		If $\kappa$ is supercompact (i.e. $\lambda$-ineffable for all $\lambda\geq\kappa$), there exists a forcing extension such that:
		\begin{enumerate}
			\setcounter{enumi}{2}
			\item $\ISP(\tau^+,\kappa,\geq\kappa)$ holds and is indestructible under ${<}\,\kappa$-directed closed forcing.
		\end{enumerate}
	\end{mysenx}
	
	The paper is organized as follows: After giving preliminaries and reviewing facts and definitions from \cite{JakobDisjointInterval}, we prove Theorem A. Afterwards we give the criterion for when a forcing order preserves $\ISP$. In the last sections, we prove the consistency results mentioned above.
	
	\begin{myack}
		This paper forms a part of the author's PhD thesis. The author would like to thank his advisor, Heike Mildenberger, for many fruitful discussions surrounding this material. He would also like to thank Maxwell Levine for alerting him to the open questions guiding this paper and discussing them with him. Lastly, he would like the thank the referee for their thorough report and many helpful comments.
	\end{myack}

	\section{Preliminaries}
	
	We assume the reader is familiar with the basics of forcing. Good introductory material can be found in the books by Jech (cf. \cite{JechSetTheory}) and Kunen (cf. \cite{KunenSetTheory}). An introduction into techniques regarding large cardinals can be found in the book by Kanamori (cf. \cite{KanamoriHigherInfinite}).
	
	\subsection*{Basic Forcing Facts}\hfill
	
	Our notation is fairly standard. We follow the convention that filters are upward closed, so that $q\leq p$ means that $q$ forces at least as much as $p$. $\dP\uhr p$ is defined as $\{q\in \dP\;|\;q\leq p\}$. $V[\dP]$ denotes an arbitrary extension by $\dP$, i.e. ``$V[\dP]$ has property $P$`` means that for every $\dP$-generic $G$, $V[G]$ has property $P$. Lastly, for any forcing order $\dP$, we let $\Gamma_{\dP}$ (omitting $\dP$ when it is clear which order we are talking about) be the canonical $\dP$-name such that $\Gamma_{\dP}^G=G$ whenever $G$ is a $\dP$-generic filter.
	
	\begin{mydef}
		Let $V\subseteq W$ be models of set theory with the same ordinals and $\kappa$ a regular cardinal in $V$. $(V,W)$ has the \emph{${<}\,\kappa$-covering property} if for any set $x\in W$, $x\subseteq V$ of size ${<}\,\kappa$ in $W$ there is $y\in V$ of size ${<}\,\kappa$ in $V$ such that $x\subseteq y$. A forcing order $\dP$ has the $\kappa$-covering property if $1_{\dP}$ forces that $(V,V[\dP])$ has it.
	\end{mydef}
	
	If $(V,W)$ has the ${<}\,\kappa$-covering property, where $\kappa$ is a regular cardinal, $\kappa$ remains a regular cardinal in $W$: Otherwise there exists a cofinal subset of $\kappa$ in $W$ of size ${<}\,\kappa$. This set is covered by an element of $V$ of size ${<}\,\kappa$, contradicting the regularity of $\kappa$. If $\dP$ is $\kappa$-cc. and $\dot{f}$ any $\dP$-name for a function from an ordinal into $V$, then there are fewer than $\kappa$ possibilities for any value $\dot{f}(\check{\alpha})$, so we obtain:
	
	\begin{myfact}
		If $\dP$ has the $\kappa$-cc., $\dP$ has the $\kappa$-covering property.
	\end{myfact}
	
	When arguing the preservation of properties which are downwards absolute, we will frequently make use of projections, which are a way of showing that an extension by some order $\dQ$ is contained in an extension by another order $\dP$.
	
	\begin{mydef}
		Let $\dP$ and $\dQ$ be forcing orders. A function $\pi\colon\dP\to\dQ$ is a \emph{projection} if the following hold:
		\begin{enumerate}
			\item $\pi(1_{\dP})=1_{\dQ}$.
			\item For all $p\leq q$, $\pi(p)\leq \pi(q)$
			\item For all $p\in\dP$, if $q\leq \pi(p)$, there is some $p'\leq p$ such that $\pi(p')\leq q$.
		\end{enumerate}
	\end{mydef}
	
	If there exists a projection from $\dP$ to $\dQ$, any extension by $\dQ$ can be forcing extended to an extension by $\dP$.
	
	\begin{mydef}
		Let $\dP$ and $\dQ$ be forcing orders, $\pi\colon\dP\to\dQ$ a projection. Let $H$ be $\dQ$-generic. In $V[H]$, the forcing order $\dP/H$ consists of all $p\in\dP$ such that $\pi(p)\in H$. We let $\dP/\dQ$ be a $\dQ$-name for $\dP/\dot{H}$ and call $\dP/\dQ$ the \emph{quotient forcing} of $\dP$ and $\dQ$.
	\end{mydef}
	
	\begin{myfact}
		Let $\dP$ and $\dQ$ be forcing orders and $\pi\colon\dP\to\dQ$ a projection. If $H$ is $\dQ$-generic over $V$ and $G$ is $\dP/H$-generic over $V[H]$, then $G$ is $\dP$-generic over $V$ and $H\subseteq\pi[G]$. In particular, $V[H][G]=V[G]$.
	\end{myfact}
	
	The approximation property was used implicitly by Mitchell in his proof that the tree property can consistently hold at successor cardinals. Later it was defined explicitly by Hamkins in \cite{HamkinsGap}:
	
	\begin{mydef}
		Let $V\subseteq W$ be models of set theory and $\delta$ a cardinal in $V$. $(V,W)$ has the \emph{${<}\,\delta$-approximation property} if $x\in V$ holds whenever $x\in W$ is such that $x\cap z\in V$ for every $z\in([x]^{<\delta})^V$. A poset $\dP$ has the ${<}\,\delta$-approximation property if $(V,V[G])$ has it for every $\dP$-generic filter $G$.
	\end{mydef}
	
	In his paper introducing the approximation property, Hamkins also showed that iterations $\dP*\dQ$ where $\dP$ is of size $<\delta$ and $\dQ$ is ${<}\,\delta$-strategically closed have the ${<}\,\delta$-approximation property. Over the years, this criterion has been improved step by step (and criteria of a very different flavor such as strong properness have been introduced). In this paper, we will obtain a further weakening of the necessary assumptions.
	
	$(\kappa,\lambda)$-lists were introduced by Weiss in his PhD thesis. We introduce both notions of ``thinness`` for completeness although we will only work with slender lists.
	
	\begin{mydef}
		Let $\kappa\leq\lambda$ be regular cardinals.
		\begin{enumerate}
			\item A \emph{$(\kappa,\lambda)$-list} is a function $f\colon[\lambda]^{<\kappa}\to[\lambda]^{<\kappa}$ such that $f(a)\subseteq a$ for every $a\in[\lambda]^{<\kappa}$
			\item A $(\kappa,\lambda)$-list $f$ is \emph{thin} if for every $x\in[\lambda]^{<\kappa}$,
			$$|\{f(z)\cap x\;|\;x\subseteq z\in[\lambda]^{<\kappa}\}|\;<\kappa$$
			\item For a cardinal $\delta\leq\kappa$, a $(\kappa,\lambda)$-list $f$ is \emph{$\delta$-slender} if for every sufficiently large $\Theta$ there exists a club $C\subseteq[H(\Theta)]^{<\kappa}$ such that for every $M\in C$:
			$$\forall z\in[\lambda]^{<\delta}\cap M\;f(M\cap\lambda)\cap z\in M$$
			\item If $f$ is a $(\kappa,\lambda)$-list, $b\subseteq\lambda$ is an \emph{ineffable branch} for $f$ if the set
			$$\{z\in[\lambda]^{<\kappa}\;|\;z\cap b=f(z)\}$$
			is stationary in $[\lambda]^{<\kappa}$.
		\end{enumerate}
		For $\delta\leq\kappa$, the statement $\ISP(\delta,\kappa,\lambda)$ says that every $\delta$-slender $(\kappa,\lambda)$-list has an ineffable branch.
	\end{mydef}
	
	\section{Strong Distributivity and Orders on Products}
	
	We expand the framework for working with arbitrary orders on products of sets by a new property and prove a result on when a forcing has the ${<}\,\delta$-approximation property.
	
	We first review definitions and results from \cite{JakobDisjointInterval}.
	
	\subsection*{Strong Distributivity}\hfill
	
	\begin{mydef}
		Let $\dP$ be a poset and $\delta$ a cardinal. $\dP$ is \emph{strongly ${<}\,\delta$-distributive} if for every $p\in\dP$ and all sequences $(D_{\alpha})_{\alpha<\delta}$ of open dense subsets of $\dP\uhr p$ there is a descending sequence $(p_{\alpha})_{\alpha<\delta}$ such that for every $\alpha<\delta$, $p_{\alpha}\in D_{\alpha}$.
	\end{mydef}
	
	Strong distributivity can be seen as a kind of uniform distributivity. It is related to the completeness game on a partial order which we will define now:
	
	\begin{mydef}
		Let $\dP$ be a forcing order, $\delta$ an ordinal. The \emph{completeness game} $G(\dP,\delta)$ on $\dP$ with length $\delta$ has players COM (complete) and INC (incomplete) playing elements of $\dP$ with COM playing at even ordinals (and limits) and INC playing at odd ordinals. COM starts by playing $1_{\dP}$, afterwards $p_{\alpha}$ has to be a lower bound of $(p_{\beta})_{\beta<\alpha}$. INC wins if COM is unable to play at some point $<\delta$. Otherwise, COM wins.
	\end{mydef}
	
	A forcing order $\dP$ is ${<}\,\kappa$-distributive if and only if for all $\alpha<\kappa$, INC does not have a winning strategy in $G(\dP,\alpha)$. The proof adapts to show the following:
	
	\begin{mysen}[Theorem 3.6 in \cite{JakobDisjointInterval}]
		$\dP$ is strongly ${<}\,\kappa$-distributive if and only if INC does not have a winning strategy in $G(\dP,\kappa)$.
	\end{mysen}
	
	An stronger property is $\kappa$-strategic closure:
	
	\begin{mydef}
		A poset $\dP$ is \emph{$\kappa$-strategically closed} if COM has a winning strategy in $G(\dP,\kappa)$. It is \emph{${<}\,\kappa$-strategically closed} if for all $\alpha<\kappa$, COM has a winning strategy in $G(\dP,\alpha)$.
	\end{mydef}
	
	It is clear that if $\dP$ is $\kappa$-strategically closed it is strongly ${<}\,\kappa$-distributive. However, there is no provable connection between ${<}\,\kappa$-strategic closure and strong ${<}\,\kappa$-distributivity (see Example 3.11 and Example 3.12 in \cite{JakobDisjointInterval}).
	
	\subsection*{Orders on Products}
	
	As variants of Mitchell Forcing are neither exactly products nor iterations, it is helpful to think of them as abstract orders on products of sets with certain regularity properties. We note that we will not assume that our partial orders are antisymmetric, i.e. for a partial order $\dP$ there can be $p,p'\in\dP$ such that $p\leq p'\leq p$ but $p\neq p'$. This situation often occurs in iterated forcing, since, if $\tau$ and $\tau'$ are different $\dP$-names but $p\Vdash\tau=\tau'$, we have $(p,\tau)\neq(p,\tau')$ but $(p,\tau)\leq(p,\tau')\leq(p,\tau)$.
	
	\begin{mydef}\label{OrdersOnProducts}
		Let $\dP$ and $\dQ$ be nonempty sets and $R$ a partial order on $\dP\times\dQ$. We define the following related partial orders:
		\begin{enumerate}
			\item The \emph{base ordering} $b(R)$ is an ordering on $\dP$ given by $p(b(R))p'$ if there are $q_0,q_1\in\dQ$ such that $(p,q_0)R(p',q_1)$.
			\item The \emph{term ordering} $t(R)$ is an ordering on $\dP\times\dQ$ given by $(p,q)(t(R))(p',q')$ if $(p,q)R(p',q')$ and $p=p'$.
			\item For $p\in\dP$, the \emph{section ordering} $s(R,p)$ is an ordering on $\dQ$ given by $q(s(R,p))q'$ if $(p,q)R(p,q')$.
		\end{enumerate}
		We also fix the following properties:
		\begin{enumerate}[label=(\roman*)]
			\item $(\dP\times\dQ,R)$ is \emph{based} if for all $p,p'$,
			$$\exists q_0,q_1((p,q_0)R(p',q_1)\longrightarrow\forall q((p,q)R(p',q)))$$
			\item $(\dP\times\dQ,R)$ has the \emph{projection property} if whenever $(p',q')R(p,q)$, there is $q''\in\dQ$ such that $(p,q'')R(p,q)$ and $(p',q'')R(p',q')R(p',q'')$.
			\item $(\dP\times\dQ,R)$ has the \emph{refinement property} if $p'(b(R)) p$ implies that $s(R,p')$ refines $s(R,p)$, i.e. whenever $(p,q')R(p,q)$ and $p'(b(R)) p$, also $(p',q')R(p',q)$.
			\item $(\dP\times\dQ,R)$ has the \emph{mixing property} if whenever $(p,q_0),(p,q_1)R(p,q)$, there are $p_0,p_1\in\dP$ and $q'\in\dQ$ with $(p,q') R (p,q)$ such that $(p_i,q') R (p,q_i)$ for $i=0,1$.
		\end{enumerate}
		We say that $(\dP\times\dQ,R)$ is \emph{iteration-like} if $(\dP\times\dQ,R)$ is based and has the projection property, the refinement property and the mixing property.
	\end{mydef}
	
	The projection and refinement property hold in almost all cases, and always for iterations and products. They are necessary for most of the relevant techniques. The mixing property roughly states that we can mix elements of $\dQ$ modulo $\dP$ and holds e.g. in iterations $\dP*\dot{\dQ}$ if $\dP$ is atomless:
	
	\begin{myex}
		Let $(\dP,\leq_{\dP})$ be a partial order and let $(\dot{\dQ},\leq_{\dot{\dQ}})$ be a $\dP$-name for a partial order. Let $\tilde{\dQ}$ be the set consisting of all $\dP$-names for elements of $\dot{\dQ}$ (technically, $\tilde{\dQ}$ is a class but we can take a sufficiently large set $\tilde{\dQ}$ such that every $\dP$-name for an element of $\dot{\dQ}$ has an equivalent $\dP$-name in $\tilde{\dQ}$). Then the iteration $\dP*\dot{\dQ}$ is equivalent to an order $R$ on the product $\dP\times\tilde{\dQ}$ if we let $(p',\dot{q}') R (p,\dot{q})$ if and only if $p'\leq_{\dP} p$ in $\dP$ and $p'\Vdash\dot{q}'\leq_{\dot{\dQ}}\dot{q}$. Let us explicitely calculate the related partial orders and check if $\dP\times\tilde{\dQ}$ has the defined properties:
		\begin{enumerate}
			\item The base ordering is simply equivalent to $(\dP,\leq_{\dP})$ since $(p,\dot{q}_0)R(p',\dot{q}_1)$ implies $p\leq_{\dP}p'$ and whenever $p\leq_{\dP}p'$, we can choose $\dot{q}_0=\dot{q}_1$.
			\item The term ordering is equivalent to the usual notion of the term ordering on an iteration as defined by Krueger (see the discussion before \cite[Proposition 2.1]{KruegerMitchellStyle}).
			\item For $p=1_{\dP}$, the section ordering $s(R,p)$ is equal to the termspace forcing as defined by Laver (see e.g. \cite[Chapter 22]{CummingsHandbook}).
		\end{enumerate}
		Now let us verify that $(\dP\times\tilde{\dQ},R)$ has the defined properties:
		\begin{enumerate}[label=(\roman*)]
			\item $(\dP\times\tilde{\dQ},R)$ is based, since $(p,\dot{q}_0)R(p',\dot{q}_1)$ implies $p\leq_{\dP}p'$ and thus that $(p,\dot{q})R(p',\dot{q})$ for any $\dot{q}\in\tilde{\dQ}$.
			\item $(\dP\times\tilde{\dQ},R)$ has the projection property: Let $(p',\dot{q}')R(p,\dot{q})$. Using standard arguments on names, let $\dot{q}''$ be a $\dP$-name for an element of $\dot{\dQ}$ which is forced by $p'$ to be equal to $\dot{q}'$ and by elements incompatible with $p'$ to be equal to $\dot{q}$. Then clearly $(p,\dot{q}'') R(p,\dot{q})$ and $(p,\dot{q}'')$ and $(p,\dot{q}')$ are equivalent.
			\item $(\dP\times\dQ,R)$ has the refinement property, since stronger conditions force more.
			\item $(\dP\times\dQ,R)$ has the mixing property if $\dP$ is atomless: Let $(p,\dot{q}_0),(p,\dot{q}_1)R(p,\dot{q})$. Since $\dP$ is atomless, we can find $p_0,p_1\leq p$ which are incompatible. Using standard arguments on names, we can find $\dot{q}'$ such that $p_0\Vdash\dot{q}'=\dot{q}_0$ and $p_1\Vdash\dot{q}'=\dot{q}_1$. Assume that conditions incompatible with both $p_0$ and $p_1$ force $\dot{q}'=\dot{q}$. Then clearly $\dot{q}'$ is as required.
		\end{enumerate}
	\end{myex}
	
	Let us note that the term ordering is the disjoint union of the section orderings, so if the term ordering is ${<}\,\delta$-closed (strategically closed, strongly distributive etc.) for some $\delta$, so are all the section orderings (and vice versa).
	
	\begin{mylem}\label{PPImpliesProjection}
		If $(\dP\times\dQ,R)$ is based and has the projection and refinement property, the identity is a projection from $(\dP,b(R))\times(\dQ,s(R,1_{\dP}))$ onto $(\dP\times\dQ,R)$.
	\end{mylem}
	
	\begin{proof}
		Denote by $R_{\pi}$ the order on $(\dP,b(R))\times(\dQ,s(R,1_{\dP}))$. If $(p',q')R_{\pi}(p,q)$, $p' (b(R))p$ and $(1,q')R(1,q)$. By the refinement property, $(p,q')R(p,q)$. By basedness $(p',q')R(p,q')$. In summary,
		$$(p',q')R(p,q')R(p,q)$$
		Assume $(p',q')R(p,q)$. This implies $p' (b(R)) p$. Furthermore, by basedness $(p,q)R(1_{\dP},q)$. By the projection property there is $q''$ such that $(1_{\dP},q'')R(1_{\dP},q)$ and $(p',q'')R(p',q')R(p',q'')$. So $(p',q'')R_{\pi}(p,q)$ and $(p',q'')R(p',q')$.
	\end{proof}
	
	If $(\dP\times\dQ,R)$ is an ordering on a product, forcing with $\dP\times\dQ$ can also be regarded as first forcing with $\dP$ and then with a special ordering on $\dQ$. However, this will not be of use to us in this paper.
	
	In \cite{UngerFragIndestructII}, Unger showed that an iteration $\dP*\dQ$, where $\dP$ is square-$\delta$-cc. and $\dQ$ is forced to be ${<}\,\delta$-closed, has the ${<}\,\delta$-approximation property. We will improve this result in two ways: Firstly, we work with an arbitrary order on a product, allowing us to apply the result more easily to variants of Mitchell Forcing and secondly, we only need a ``distributivity`` assumption as opposed to one about closure. This is important, as, unlike ${<}\,\delta$-closure, strong ${<}\,\delta$-distributivity is preserved by forcing with $\delta$-cc. forcings. So in particular, if we have a forcing $(\dP\times\dQ,R)$ with its base ordering isomorphic to $\Add(\omega,X)$ for some set $X$ and a strongly ${<}\,\omega_1$-distributive term ordering, it will satisfy the conclusion of Lemma \ref{ApproxProp} in any ccc.\ extension. This is a slight improvement of Lemma 5.4 in \cite{UngerSuccessiveApproach}, where the square-ccc.\ is required instead.
	
	\begin{mysen}\label{ApproxProp}
		Let $(\dP\times\dQ,R)$ be an iteration-like partial order and $\delta$ a cardinal. Assume $(\dP,b(R))^2$ is $\delta$-cc. and $(\dP\times\dQ,t(R))$ is strongly $<\delta$-distributive. Then $(\dP\times\dQ,R)$ has the $<\delta$-approximation property.
	\end{mysen}
	
	We begin with a helping lemma:
	
	\begin{mylem}\label{DecisionByP}
		Let $(\dP\times\dQ,R)$ be an iteration-like partial order. If $(p,q)$ forces $\dot{x}\in V$ but for any $y\in V$, $(p,q)\not\Vdash\dot{x}=\check{y}$, there are $q''\in\dQ$, $p_0,p_1 b(R) p$ and $y_0\neq y_1$ such that $(p,q'')R (p,q)$ and for $i\in 2$, $(p_i,q'')\Vdash\dot{x}=\check{y}_i$.
	\end{mylem}
	
	\begin{proof}
		For better readability, we prove the result in a series of statements, showing where we apply which property. We let PP stand for the projection property, RP for the refinement property and MP for the mixing property. Let BS stand for basedness. We check two cases:
		
		Assume there exist $q_0,y_0$ such that $(p,q_0)R(p,q)$ and $(p,q_0)\Vdash\dot{x}=\check{y}_0$.
		\begin{align*}
			(1)\;\; & \exists q_0,y_0(((p,q_0) R (p,q))\wedge ((p,q_0)\Vdash\dot{x}=\check{y}_0)) \\
			\;\;& \exists (p',q'''),y_1\neq y_0 (((p',q''') R (p,q))\wedge ((p',q''')\Vdash\dot{x}=\check{y}_1) \\
			(PP)\;\; & \exists q_1((p,q_1) R (p,q)\wedge ((p',q_1) R (p',q'''))) \\
			\;\;& (p',q_1)\Vdash\dot{x}=\check{y}_1 \\
			\;\;& p' (b(R)) p  \\
			(BS)\;\;& ((p',q_0) R (p,q_0)) \\
			\;\;& (p', q_0)\Vdash\dot{x}=\check{y}_0 \\
			(RP)\;\; & ((p',q_0)R(p',q))\wedge((p',q_1)R(p',q)) \\
			(MP)\;\; & \exists p_0,p_1,q'(((p',q') R (p',q))\wedge((p_0,q') R (p',q_0))\wedge((p_1,q') R (p',q_1))) \\
			\;\;& p_0,p_1 (b(R)) p' \\
			(BS)\;\; & (p',q') R (p',q) R (p,q) \\
			(PP)\;\; & \exists q'' ((p,q'') R (p,q)) \wedge ((p',q'') R (p',q') R (p',q'')) \\
			(RP)\;\; & ((p_0,q'') R (p_0,q') R (p',q_0))\wedge((p_1,q'') R (p_1,q') R (p',q_1)) \\
			\;\;& ((p_0,q'')\Vdash\dot{x}=\check{y}_0)\wedge((p_1,q'')\Vdash\dot{x}=\check{y}_1)
		\end{align*}
		
		Assume case (1) does not hold, i.e. for all $q_0,y_0$ with $(p,q_0) R(p,q)$, $(p,q_0)\not\Vdash\dot{x}=\check{y}_0$.
		\begin{align*}
			(2)\;\; & \forall q_0,y_0(((p,q_0) R (p,q))\rightarrow((p,q_0)\not\Vdash\dot{x}=\check{y}_0)) \\
			\;\; & \exists (p_0,q'),y_0 (((p_0,q')R(p,q))\wedge((p_0,q')\Vdash\dot{x}=\check{y}_0)) \\
			(PP) \;\; & \exists q'''(((p,q''') R (p,q))\wedge ((p_0,q''') R (p_0,q') R (p_0,q'''))) \\
			\;\; & ((p,q''')\not\Vdash\dot{x}=\check{y}_0) \\
			\;\; & \exists(p_1,q''''), y_1\neq y_0(((p_1,q'''') R (p,q'''))\wedge((p_1,q'''')\Vdash\dot{x}=\check{y}_1)) \\
			(PP)\;\; & \exists q''((p,q'') R(p,q''')\wedge((p_1,q'') R (p_1,q'''') R (p_1,q''))) \\
			\;\; & ((p_1,q'')\Vdash\dot{x}=\check{y}_1) \\
			\;\; & p_0,p_1 (b(R)) p \\
			\;\; & (p,q'')R(p,q''')R(p,q) \\
			(RP)\;\; & (p_0,q'') R (p_0,q''') R (p_0,q') \\
			\;\; & ((p_0,q'')\Vdash\dot{x}=\check{y}_0)
		\end{align*}
		
	\end{proof}
	
	Now we can finish the proof of Theorem \ref{ApproxProp}.
	
	\begin{proof}[Proof of Theorem \ref{ApproxProp}]
		Let $\dot{f}$ be a $(\dP\times\dQ,R)$-name for a function such that some $(p,q)$ forces $\dot{f}\notin V$ and $\dot{f}\uhr \check{u}\in V$ for every $u\in[V]^{<\delta}\cap V$. We will construct a winning strategy for INC in the completeness game of length $\delta$ played on $(\dQ,s(R,p))\uhr q$. In any run $(q_{\gamma})_{\gamma\in\delta}$ of the game, we will construct $(p_{\gamma}^0,p_{\gamma}^1,y_{\gamma})_{\gamma\in\Odd\cap\delta}$ such that
		\begin{enumerate}
			\item $y_{\gamma}\in [V]^{<\delta}\cap V$ and the sequence $(y_{\gamma})_{\gamma\in\Odd\cap\delta}$ is $\subseteq$-increasing
			\item $p_{\gamma}^0,p_{\gamma}^1 b(R) p$
			\item $(p_{\gamma}^0,q_{\gamma})$ and $(p_{\gamma}^1,q_{\gamma})$ decide $\dot{f}\uhr \check{y}_{\alpha}$ equally for any odd $\alpha<\gamma$, but differently for $\alpha=\gamma$
		\end{enumerate}
		Assume the game has been played until some even ordinal $\gamma<\delta$. Let $y_{\gamma+1}':=\bigcup_{\alpha\in\gamma\cap\Odd}y_{\alpha}$, which has size $<\delta$. $\dot{f}\uhr y_{\gamma+1}'$ is forced to be in $V$, so we can find $(p_{\gamma+1}',q_{\gamma+1}') R (p,q_{\gamma})$ which decides $\dot{f}\uhr y_{\gamma+1}'$. By the projection property, we can find $q_{\gamma+1}''$ such that $(p,q_{\gamma+1}'') R (p,q_{\gamma})$ and $(p_{\gamma+1}',q_{\gamma+1}'') R (p_{\gamma+1}',q_{\gamma+1}')$, so $(p_{\gamma+1}',q_{\gamma+1}'')$ also decides $\dot{f}\uhr\check{y}_{\gamma+1}'$. Because $\dot{f}$ is forced to be outside of $V$, there is $\beta$ such that $(p_{\gamma+1}',q_{\gamma+1}'')$ does not decide $\dot{f}(\check{\beta})$. Let $y_{\gamma+1}:=y_{\gamma+1}'\cup\{\beta\}$. Then $(p_{\gamma+1}',q_{\gamma+1}'')$ does not decide $\dot{f}\uhr\check{y}_{\gamma+1}$, so we find $q_{\gamma+1}'''$ and $p_{\gamma+1}^0,p_{\gamma+1}^1$ such that $(p_{\gamma+1}',q_{\gamma+1}''') R (p_{\gamma+1}',q_{\gamma+1}'')$ and $(p_{\gamma+1}^0,q_{\gamma+1}''')$ and $(p_{\gamma+1}^1,q_{\gamma+1}''')$ decide $\dot{f}\uhr \check{y}_{\gamma+1}$ differently. Lastly, use the projection property to obtain $q_{\gamma+1}$ such that $(p,q_{\gamma+1}) R (p,q_{\gamma+1}'')$ and $(p_{\gamma+1}',q_{\gamma+1}) R (p_{\gamma+1}',q_{\gamma+1}''')$. It follows that these objects are as required.
		
		Lastly, assume this strategy does not win, i.e. there is a game of length $\delta$. In this case, we claim that $\{(p_{\gamma}^0,p_{\gamma}^1)\;|\;\gamma\in\Odd\}$ is an antichain in $(\dP,b(R))^2$, obtaining a contradiction. Assume $(p^0,p^1) (b(R)^2) (p_{\gamma}^0,p_{\gamma}^1),(p_{\gamma'}^0,p_{\gamma'}^1)$ with $\gamma'<\gamma$. Because $p^0 (b(R)) p_{\gamma}^0$ and $p^1 (b(R)) p_{\gamma}^1$, $(p^0,q_{\gamma}) R (p_{\gamma}^0,q_{\gamma})$ and $(p^1,q_{\gamma}) R(p_{\gamma}^1,q_{\gamma})$ decide $\dot{f}\uhr\check{y}_{\gamma'}$ equally, but because $p^0 (b(R)) p_{\gamma'}^0$ and $p^1 (b(R)) p_{\gamma'}^1$, $(p^0,q_{\gamma}) R(p_{\gamma'}^0,q_{\gamma}) R (p_{\gamma'}^0, q_{\gamma'})$ (using the refinement property) and $(p^1,q_{\gamma}) R(p_{\gamma'}^1,q_{\gamma})$ $R (p_{\gamma'}^1,q_{\gamma'})$ decide $\dot{f}\uhr\check{y}_{\gamma'}$ differently, a contradiction.
	\end{proof}
	
	We obtain an interesting statement regarding strongly distributive forcings (which contrasts the fact that consistently there can exist a ${<}\,\omega_1$-distributive, ccc. poset).
	
	\begin{mycol}
		Let $\dP$ be a nontrivial poset and $\delta$ a cardinal. Then one of the following holds:
		\begin{enumerate}
			\item $\dP$ is not $\delta$-cc.
			\item $\dP$ is not strongly ${<}\,\delta$-distributive.
		\end{enumerate}
	\end{mycol}
	
	\begin{proof}
		Assume to the contrary that none of the above holds. As $\dP$ is nontrivial, it adds some set $x$ of ordinals. Because $\dP$ is in particular ${<}\,\delta$-distributive, $x\cap y\in V$ for every $y\in[x]^{<\delta}\cap V$. Thus $\dP$ does not have the ${<}\,\delta$-approximation property. On the other hand, $\dP$ preserves the $\delta$-cc. of itself by virtue of being strongly ${<}\,\delta$-distributive (by \cite{JakobDisjointInterval}, Lemma 3.7), so $\dP\times\dP$ has the $\delta$-cc. and $\dP$ has the ${<}\,\delta$-approximation property (by viewing $\dP$ as an order on the product $\dP\times\{0\}$), a contradiction.
	\end{proof}
	
	\section{Ineffability Witnesses}
	
	We will now relate the existence of ineffable branches for slender lists to the existence of powerful elementary submodels of $H(\Theta)$, similarly to \cite{HolyLNSmallEmbeddingLargeCard}.
	
	\begin{mysen}\label{IneffChara}
		Let $\kappa<\lambda$ be cardinals. The following are equivalent:
		\begin{enumerate}
			\item $\ISP(\kappa,\lambda)$ holds.
			\item For all sufficiently large cardinals $\Theta$ and every $(\kappa,\lambda)$-list $f$, there is $M\prec H(\Theta)$ of size ${<}\,\kappa$ with $M\cap\kappa\in\kappa$ such that $\kappa,\lambda,f\in M$ and for some $b\in M$, $b\cap M=f(M\cap\lambda)$.
		\end{enumerate}
	\end{mysen}
	
	\begin{proof}
		Assume $\kappa$ is $\lambda$-ineffable and $f$ is a $(\kappa,\lambda)$-list. By $\lambda$-ineffability there is $b\subseteq\lambda$ such that the set
		$$S:=\{x\in[\lambda]^{<\kappa}\;|\;f(x)=b\cap x\}$$
		is stationary in $[\lambda]^{<\kappa}$. Using standard techniques we find $M\prec H(\Theta)$ with $\kappa,\lambda,f,b\in M$ such that $M\cap\kappa\in\kappa$ and $M\cap\lambda\in S$. Clearly $b\cap M=f(M\cap\lambda)$.
		
		Assume the submodel property holds and let $f$ be a $(\kappa,\lambda)$-list. Find $M\prec H(\Theta)$ of size $<\kappa$ such that $\kappa,\lambda,f\in M$ and for some $b\in M$, $b\cap M=f(M\cap\lambda)$. We will show that $b$ is an ineffable branch for $f$. If $b$ is not an ineffable branch, by elementarity, there is $C\in M$ club in $[\lambda]^{<\kappa}$ such that for every $x\in C$, $f(x)\neq x\cap b$. $M$ contains a function $F$ such that every $x\in[\lambda]^{<\kappa}$ closed under $F$ is in $C$. Because $F\in M\prec H(\Theta)$, $M\cap\lambda$ is closed under $F$, so $M\cap\lambda\in C$ but $f(M\cap\lambda)=M\cap b$, a contradiction.
	\end{proof}
	
	Note that the previous Theorem works ``list-by-list'', i.e. for any $(\kappa,\lambda)$-list $f$, $f$ has an ineffable branch if and only if there is $M\prec H(\Theta)$ of size ${<}\,\kappa$ with $M\cap\kappa\in\kappa$ such that $\kappa,\lambda,f\in M$ and for some $b\in M$, $b\cap M=f(M\cap\lambda)$.
	
	We fix a definition corresponding to the previous Lemma:
	
	\begin{mydef}
		Let $\kappa\leq\lambda$ be cardinals and $f$ a $(\kappa,\lambda)$-list. Let $\Theta$ be large. $M\in[H(\Theta)]^{<\kappa}$ is a \emph{$\lambda$-ineffability witness for $\kappa$ with respect to $f$} if $M\prec H(\Theta)$, $M\cap\kappa\in\kappa$, $\{f,\kappa,\lambda\}\subseteq M$, and there is $b\in M$ such that $b\cap M=f(M\cap\lambda)$.
	\end{mydef}
	
	So in particular, we have shown that a $(\kappa,\lambda)$-list has an ineffable branch if and only if there is a $\lambda$-ineffability witness for $\kappa$ with respect to $f$. This enables us to show $\ISP(\delta,\kappa,\lambda)$ in the following way: Given a name $\dot{f}$ for a $\delta$-slender $(\kappa,\lambda)$-list, we construct (in the ground model) a related $(\kappa,\lambda)$-list $e$ (using the slenderness of $\dot{f}$) such that, given a $\lambda$-ineffability witness $M$ for $\kappa$ with respect to $e$ in the ground model, $M[G]$ is a $\lambda$-ineffability witness for $\kappa$ with respect to $\dot{f}^G$.
	
	Later, we will use witnesses with additional properties:
	
	\begin{mylem}\label{IneffEmbed}
		Let $\kappa<\lambda=\lambda^{<\kappa}$ be cardinals and assume $\kappa$ is $\lambda$-ineffable. For all sufficiently large cardinals $\Theta$, every $x\in[H(\Theta)]^{<\kappa}$ and every $(\kappa,\lambda)$-list $f$, there is $M\prec H(\Theta)$ such that the following holds:
		\begin{enumerate}
			\item $M\cap\kappa\in\kappa$ is inaccessible and $[M\cap\lambda]^{<M\cap\kappa}\subseteq M$
			\item $x\subseteq M$
			\item $\kappa,\lambda,f\in M$ and for some $b\in M$, $b\cap M=f(M\cap\lambda)$.
		\end{enumerate}
	\end{mylem}
	
	\begin{proof}
		Let $F\colon\lambda\to[\lambda]^{<\kappa}$ be a bijection. We modify $f$ to obtain the following $(\kappa,\lambda)$-list $g$:
		\begin{enumerate}
			\item $g(a):=\emptyset$ if $a\cap\kappa$ is not an ordinal.
			\item $g(a)$ is a cofinal subset of $a\cap\kappa$ of ordertype ${<}\,a\cap\kappa$ if $a\cap\kappa$ is a singular ordinal.
			\item $g(a)$ is an element of $[a]^{<a\cap\kappa}\smallsetminus F[a]$ if $a\cap\kappa$ is a regular ordinal and such an element exists.
			\item $g(a):=f(a)$ otherwise.
		\end{enumerate}
		Since $\kappa$ is $\lambda$-ineffable, there is $M\prec H(\Theta)$ of size ${<}\,\kappa$ such that $\kappa,\lambda,f\in M$, $M\cap\kappa\in\kappa$ and for some $b\in M$, $M\cap b=f(M\cap\lambda)$. By the proof of Theorem \ref{IneffChara} we can assume $x\subseteq M$ and that $M\cap\kappa$ is a strong limit. We will verify that only the last case can hold for $a:=M\cap\lambda$ (with the first case being excluded by assumption).
		\begin{enumerate}
			\setcounter{enumi}{1}
			\item Assume $a\cap\kappa$ is singular. Let $g\in M$ be a function enumerating $b$ (in ascending order). Then, as $\cf(a\cap\kappa)<a\cap\kappa$, $g[\cf(a\cap\kappa)]=f(M\cap\lambda)\in M$, so $M$ says that $\kappa$ is singular, a contradiction.
			\item Assume $[a]^{<a\cap\kappa}\smallsetminus F[a]$ is nonempty. Then $g(a)$ is a subset of $M\cap\lambda$ of ordertype ${<}\,a\cap\kappa$. Proceed as before to show that $g(a)\in M$, which implies $g(a)\in F[a]$ by elementarity, a contradiction.
		\end{enumerate}
		Thus, $b\cap M=g(M\cap\lambda)=f(M\cap\lambda)$. By elementarity, $M\supseteq F[a]=[a]^{<a\cap\kappa}$, so we are done.
	\end{proof}
	
	We also need two results about the interaction of $\lambda$-ineffability witnesses and forcing: Given any set $M$ as well as any forcing order $\dP$ and $\dP$-generic filter $G$, let $M[G]$ consist of $\sigma_G$ for all $\sigma\in M$ that are $\dP$-names. Note that if $\dP\in M$ and $M\prec H(\Theta)$ for some large enough $\Theta$, $\Gamma_{\dP}\in M$ and thus $G=\Gamma_{\dP}^G\in M[G]$.
	
	The following is shown in \cite{ShelahProperImproper}, chapter III, theorem 2.11:
	
	\begin{mylem}
		If $\dP\in M\prec H(\Theta)$ and $G$ is $\dP$-generic, $M[G]\prec H(\Theta)^{V[G]}$.
	\end{mylem}
	
	We also have two additional Lemmas similar to properness.
	
	\begin{mylem}
		Let $M\prec H(\Theta)$ with Mostowski-Collapse $\pi\colon M\to N$. Let $\dP\in M$ be a forcing order and $G$ a $\dP$-generic filter over $V$. The following are equivalent:
		\begin{enumerate}
			\item $M[G]\cap V=M$
			\item Whenever $D\in M$ is open dense in $\dP$, $D\cap M\cap G\neq\emptyset$.
			\item Whenever $D\in N$ is open dense in $\pi(\dP)$, $\pi[M\cap G]\cap D\neq\emptyset$.
		\end{enumerate}
	\end{mylem}
	
	\begin{proof}
		Let us sketch a proof.
		
		Assume (1) holds. If $D\in M$ is open dense in $\dP$, let $A\subseteq D$, $A\in M$, be a maximal antichain. Then $M$ contains the name $\tau:=\{(p,\check{p})\;|\;p\in A\}$. It follows that $p:=\tau^G\in M[G]\cap V=M$, which implies that $p\in A\cap M\cap G\subseteq D\cap M\cap G$.
		
		Assume (2) holds. Whenever $D\in N$ is open dense in $\pi(\dP)$, $D=\pi(E)$ for some $E\in M$ which, by elementarity, is open dense in $\dP$. Ergo $D\cap M\cap G$ contains some $p$. Then $\pi(p)\in\pi[M\cap G]\cap\pi(E)=\pi[M\cap G]\cap D$.
		
		Assume (3) holds. Let $\tau$ be a $\dP$-name for an element of $V$ lying in $M$. Then $M$ contains the open dense set $D$ of conditions $p$ forcing $\tau=\check{x}_p$ for some $x_p$. So $\pi(D)\in N$ is open dense in $\pi(\dP)$, which implies $\pi[M\cap G]\cap\pi(D)\neq\emptyset$ and thus that $D\cap M\cap G$ contains some $p$. The corresponding $x_p$ is in $M$ as well by elementarity, so $x_p=\tau^G\in M$ (since $p\in G$).
	\end{proof}
	
	\begin{mylem}\label{ClosureAfterForcing}
		Let $M\prec H(\Theta)$ of size ${<}\,\kappa$ with Mostowski-Collapse $\pi\colon M\to N$. Assume $\nu:=|M|=M\cap\kappa\in\kappa$ and $[M\cap\lambda]^{<\nu}\subseteq M$. Let $\dP\in M$ be a poset of size $\lambda$ with the $\kappa$-cc and let $G$ be $\dP$-generic over $V$. Then the following holds:
		\begin{enumerate}
			\item $M[G]\cap V=M$,
			\item $[\pi(\lambda)]^{<\nu}\cap V[\pi[G\cap M]]\subseteq N[\pi[G\cap M]]$.
		\end{enumerate}
	\end{mylem}
	
	\begin{proof}
		We will assume that $\dP$ is a poset on $\lambda$.
		
		Let $\sigma\in M$ be such that $\sigma_G\in V$. By the forcing theorem and elementarity, there is $p\in G\cap M$ which forces $\sigma\in V$. Hence, again by elementarity, there exists $A\in M$ which is a maximal antichain in $\dP$ such that for any $q\in A$ there is $y_q$ such that $q\Vdash\sigma=\check{y}_q$. Because $M\cap\kappa\in\kappa$ and $|A|<\kappa$ by the $\kappa$-cc. of $\dP$, $A\subseteq M$ and thus $y_q\in M$ for every $q\in A$. Since $\sigma_G$ is one of the $y_q$, the statement follows.
		
		Now assume $x\in[\pi(\lambda)]^{\mu}\cap V[\pi[G\cap M]]$ for some $\mu<M\cap\kappa$. Let $\dot{x}$ be a $\pi(\dP)$-name for $x$ and for $\alpha\in\mu$, let $A_{\alpha}$ be a maximal antichain deciding the $\alpha$th value of $\dot{x}$. Then $|A_{\alpha}|<M\cap\kappa$ by the $M\cap\kappa$-cc. of $\pi(\dP)$. Let $f:\bigcup_{\alpha<\mu}\{\alpha\}\times A_{\alpha}\longrightarrow V$ be such that $f(\alpha,q)$ is the value $q$ decides for the $\alpha$th element of $\dot{x}$. Then (after some coding) $f\in[\pi(\lambda)]^{<M\cap\kappa}$, so $\pi^{-1}[f]\in M$ and $f=\pi(\pi^{-1}[f])\in N$. From $f$ and $\pi[G\cap M]$, we can recover $\dot{x}^{\pi[G\cap M]}$, so $\dot{x}^{\pi[G\cap M]}\in N[\pi[G\cap M]]$.
	\end{proof}
	
	The last result we need can be seen as dual to the lifting of elementary embeddings. This is folklore, for a proof, see \cite[Lemma 5.3.6]{JakobPhD}.
	
	\begin{mylem}\label{PiExtension}
		Let $M\prec H(\Theta)$ with Mostowski-Collapse $(N,\pi)$. Assume $\dP\in M$ is a poset, $G$ a $\dP$-generic filter and $M[G]\cap V=M$. Then the Mostowski-Collapse of $M[G]$ is given by $(N[\pi[G\cap M]],\pi_{M[G]})$, where $\pi_{M[G]}(\sigma^G)=(\pi(\sigma))^{\pi[G\cap M]}$. Furthermore, $\pi_{M[G]}\uhr M=\pi$.
	\end{mylem}
	
	We now give a criterion for when a forcing order forces $\ISP$, generalizing the result and adapting the proof from \cite[Theorem 7.2]{HolyLNSmallEmbeddingLargeCard}:
	
	\begin{mysen}\label{MainTheorem}
		Let $\delta\leq\kappa\leq\lambda=\lambda^{<\kappa}$ be regular cardinals and $\dP$ a poset. Assume the following:
		\begin{enumerate}
			\item $\kappa$ is $\lambda$-ineffable,
			\item $\dP$ is of size $\leq\lambda$ and $\kappa$-cc.,
			\item For every $(\kappa,\lambda)$-list $e$, every sufficiently large $\Theta$ and every $x\in[H(\Theta)]^{<\kappa}$, there is a $\lambda$-ineffability witness $M$ for $\kappa$ with respect to $e$ such that $\dP\in M$ and the following holds:
			\begin{enumerate}
				\item $x\subseteq M$
				\item $[M\cap\lambda]^{<M\cap\kappa}\subseteq M$.
				\item Whenever $G$ is $\dP$-generic over $V$, $\pi[G\cap M]$ is $\pi(\dP)$-generic over $V$ and the pair $(V[\pi[G\cap M]],V[G])$ has the ${<}\,\pi(\delta)$-approximation property.
			\end{enumerate}
		\end{enumerate}
		Then $\dP$ forces $\ISP(\check{\delta},\check{\kappa},\check{\lambda})$.
	\end{mysen}
	
	\begin{proof}
		Without loss of generality, consider $\dP$ to be a partial order on $\lambda$. Denote by $\langle\cdot,\cdot\rangle$ the G\"odel pairing function.
		
		Let $\dot{f}$ be a $\dP$-name for a $\delta$-slender $(\kappa,\lambda)$-list, forced by some $p$. Let $\dot{F}$ be the function corresponding to the club (in some $[H(\Theta')]^{<\kappa}$) witnessing the $\delta$-slenderness of $\dot{f}$. Let $\Theta$ be large so that $H(\Theta)$ contains all relevant objects.
		
		We will transform $\dot{f}$ into a ground-model $(\kappa,\lambda)$-list. To this end, let $a\in[\lambda]^{<\kappa}$. We consider two cases:
		\begin{enumerate}
			\item If there exists $M\prec H(\Theta)$ (with Mostowski-Collapse $\pi$) such that $M\cap\kappa\in \kappa$, $M\cap\lambda=a$, $[M\cap\lambda]^{<M\cap\kappa}\subseteq M$, $\kappa,\lambda,\dot{f},\dot{F}\in M$ and for some $\pi(\dP)$-name $\dot{x}_a$ for a subset of $\pi[M\cap\lambda]$ and a condition $p_a\leq p$, $p_a\Vdash\dot{f}(M\cap\lambda)=\pi^{-1}[\dot{x}_a^{\pi[\Gamma\cap M]}]$, let
			$$g(a):=\{\langle\alpha,\beta\rangle\;|\;\alpha,\beta\in\pi[a]\wedge\alpha\Vdash\check{\beta}\in\dot{x}_a\}$$
			which is a subset of $\pi[a]$ and 
			$$e(a):=\pi^{-1}[g(a)]$$
			\item Otherwise, let $e(a):=\emptyset$.
		\end{enumerate}
		Let $M\in[H(\Theta)]^{<\kappa}$ be a $\lambda$-ineffability witness for $\kappa$ with respect to $e$ as in the requirements. Let $\pi\colon M\to N$ be the Mostowski-Collapse and $a:=M\cap\lambda$ as well as $\theta:=\pi[a]$. We will show that case (1) holds. Let $G_0$ be $\dP$-generic containing $p$ and $G_0':=\pi[G_0\cap M]$. By Lemma \ref{ClosureAfterForcing}, $\pi$ extends to $\pi\colon M[G_0]\to N[G_0']$ and $[\theta]^{<\nu}\cap V[G_0']\subseteq N[G_0']$. Furthermore, $M[G_0]\prec H(\Theta)^{V[G_0]}$, so $M[G_0]\cap H(\Theta')$ is closed under $\dot{F}^{G_0}$ and therefore witnesses the slenderness of $\dot{f}^{G_0}$.
		
		Assume $\pi[\dot{f}^{G_0}(a)]\notin V[G_0']$. By the ${<}\,\pi(\delta)$-approximation property there is $z\in V[G_0']$ with ordertype ${<}\,\pi(\delta)$ such that $\pi[\dot{f}^{G_0}(a)]\cap z\notin V[G_0']$. We can assume $z\subseteq \pi[a]$, so we have $z\in N[G_0']$ and $\pi^{-1}(z)=\pi^{-1}[z]\in M[G_0]$. We have
		$$\pi[\dot{f}^{G_0}(a)]\cap z=\pi[\dot{f}^{G_0}(a)]\cap \pi[\pi^{-1}[z]]=\pi[\dot{f}^{G_0}(a)\cap \pi^{-1}[z]]$$
		and, since $M[G_0]\cap H(\Theta')$ witnesses the slenderness of $\dot{f}^{G_0}$, we have $\dot{f}^{G_0}(a)\cap\pi^{-1}[z]\in M[G_0]$, which implies that $\dot{f}^{G_0}(a)\cap\pi^{-1}[z]\subseteq M[G_0]$ and thus that
		$$\pi[\dot{f}^{G_0}(a)\cap\pi^{-1}[z]]=\pi(\dot{f}^{G_0}(a)\cap\pi^{-1}[z])\in N[G_0']$$
		which presents a contradiction. So $\pi[\dot{f}^{G_0}(a)]\in V[G_0']$. Ergo there is $p_a\leq p$ as well as a $\pi(\dP)$-name $\dot{x}_a$ for a subset of $\theta$ such that $p_a\Vdash\pi[\dot{f}(a)]=\dot{x}_a^{\pi[\Gamma\cap M]}$ which is what we wanted to show.
		
		Now our aim is to show that $p_a$ forces $M[\Gamma]$ to be a $\lambda$-ineffability witness for $\kappa$ with respect to $\dot{f}$. So let $G_1$ be a $\dP$-generic filter containing $p_a$. It is clear that $\kappa,\lambda,\dot{f}^{G_1}\in M[G_1]$. By assumption there is $b_e\in M$ such that $b_e\cap M=e(M\cap\lambda)$. Define
		$$b_f:=\{\beta\;|\;\exists\alpha\in G_1\;\langle\alpha,\beta\rangle\in b_e\}\in M[G_1]$$
		all that is left is to show $b_f\cap M=\dot{f}^{G_1}(M\cap\lambda)=\dot{f}^{G_1}(M[G_1]\cap\lambda)$ (where the last equality holds by Lemma \ref{ClosureAfterForcing}).
		
		Let $\beta\in b_f\cap M$. By elementarity there is $\alpha\in G_1\cap M$ such that $\langle\alpha,\beta\rangle\in b_e$ and we have $\langle\alpha,\beta\rangle\in b_e\cap M=e(a)=\pi^{-1}[g(a)]$. So $\pi(\langle\alpha,\beta\rangle)=\langle\pi(\alpha),\pi(\beta)\rangle\in g(a)$. By the definition, $\pi(\alpha)\Vdash\pi(\check{\beta})\in\dot{x}_a$. Hence $\pi(\beta)\in \dot{x}_a^{\pi[G_1\cap M]}$ and $\beta\in\dot{f}^{G_1}(a)$.
		
		Let $\beta\in \dot{f}^{G_1}(M\cap\lambda)$, so $\pi(\beta)\in\dot{x}_a^{\pi[G_1\cap M]}$. Thus there exists $\alpha'\in\pi[G_1\cap M]$ such that $\alpha'\Vdash\pi(\beta)\in\dot{x}_a$. Let $\alpha'=\pi(\alpha)$, $\alpha\in G_1\cap M$. Hence $\pi(\langle\alpha,\beta\rangle)\in g(a)$ and $\langle\alpha,\beta\rangle\in e(a)=b_e\cap M$. Ergo $\beta\in b_f\cap M$.
	\end{proof}
	
	We directly obtain an answer to problem 4.33 given by Mohammadpour in \cite{MohammadpourRoadCompactness}, namely that $\ISP$ can hold at weakly, but not strongly, inaccessible cardinals:
	
	\begin{mycol}
		Let $\tau<\kappa\leq\lambda$ be cardinals such that $\kappa$ is $\lambda$-ineffable. There exists a generic extension in which $\kappa$ is weakly inaccessible, $2^{\tau}=\kappa$ and $\ISP(\tau^+,\kappa,\lambda)$ holds.
	\end{mycol}
	
	\begin{proof}
		Let $\dP:=\Add(\tau,\kappa)$. Let $e$ be any $(\kappa,\lambda)$-list, $\Theta$ large and $x\in[H(\Theta)]^{<\kappa}$. Find a $\lambda$-ineffability witness $M$ for $\kappa$ as in Lemma \ref{IneffEmbed} with $\tau^++1\subseteq M$. It follows that $\pi(\dP)=\Add(\tau,M\cap\kappa)$. If $G$ is $\dP$-generic, $\pi[G\cap M]$ is $\pi(\dP)$-generic and $V[G]$ is an extension of $V[\pi[G\cap M]]$ by $\Add(\tau,\kappa\smallsetminus(M\cap\kappa))$. In $V[\pi[G\cap M]]$, this forcing is $\tau^+$-Knaster and thus has the ${<}\,\tau^+=\pi(\tau^+)$-approximation property (viewing it as an ordering on the product $\dP\times\{0\}$ which is clearly iteration-like).
	\end{proof}
	
	The same works for any desired value of $2^{\tau}$ (subject to the usual constraints), as long as $\kappa$ is ineffable enough.

	\section{Controlling the Slenderness at $\kappa$}
	
	We will start with the easiest construction regarding $\ISP$: In this section we will define a forcing $\dM_0(\tau,\mu,\kappa)$ which forces $2^{\tau}=\mu^+=\kappa$ and $\ISP(\tau^+,\kappa,\lambda)$ (if $\kappa$ is $\lambda$-ineffable), answering Problem 4.31 given by Mohammadpour in \cite{MohammadpourRoadCompactness}. Additionally, it will force $\neg\ISP(\tau,\kappa,\kappa)$. Thus we are able to control exactly how slender lists must be to have guaranteed ineffable branches.
	
	For this section, fix regular cardinals $\tau<\mu<\kappa$ such that $\tau^{<\tau}=\tau$ and $\kappa$ is inaccessible.
	
	For technical reasons, for any ordinal $\gamma$, we let $\Add(\tau,\gamma)$ be the ${<}\,\tau$-supported product of $\Add(\tau)$ of length $\gamma$, i.e. conditions are partial functions $p$ on $\gamma$ of size ${<}\,\tau$ such that for all $\alpha\in\dom(p)$, $p(\alpha)\in\Add(\tau)$, ordered in the usual way.
	
	\begin{mydef}\label{DefM0}
		Let $\nu\in(\mu,\kappa]$ be inaccessible. $\dM_0(\tau,\mu,\nu)$ consists of pairs $(p,q)$ such that
		\begin{enumerate}
			\item $p\in\Add(\tau,\nu)$
			\item $q$ is a ${<}\,\mu$-sized partial function on $\nu$ such that whenever $\gamma\in\dom(q)$, $\gamma=\delta+1$ for an inaccessible cardinal $\delta$ and $q(\gamma)$ is an $\Add(\tau,\gamma)$-name for a condition in $\dot{\Coll}(\check{\mu},\check{\delta})$.
		\end{enumerate}
		We define $(p',q')\leq(p,q)$ if
		\begin{enumerate}
			\item $p'\leq p$
			\item $\dom(q')\supseteq\dom(q)$ and for all $\gamma\in\dom(q)$,
			$$p'\uhr\gamma\Vdash q'(\gamma)\leq q(\gamma)$$
		\end{enumerate}
	\end{mydef}
	
	The forcing $\dM_0$ was used by Levine in \cite{LevineDisjointStatSeq} to give an easier construction of a model with a disjoint stationary sequence on an arbitrary double successor cardinal and disjoint stationary sequences on two successive cardinals.
	
	We can obviously view $\dM_0(\tau,\mu,\nu)$ as an ordering on some product $\Add(\tau,\nu)\times\dT(\tau,\mu,\nu)$, which is based.
	
	\begin{mylem}\label{M0Prop}
		Let $\nu\in(\mu,\kappa]$ be inaccessible.
		\begin{enumerate}
			\item $\dM_0(\tau,\mu,\nu)$ is ${<}\,\tau$-directed closed.
			\item $\dM_0(\tau,\mu,\nu)$ is $\nu$-Knaster,
			\item The base ordering on $\dM_0(\tau,\mu,\nu)$ is $\tau^+$-Knaster,
			\item The term ordering on $\dM_0(\tau,\mu,\nu)$ is ${<}\,\mu$-closed,
			\item The ordering is iteration-like.
		\end{enumerate}
	\end{mylem}
	
	\begin{proof}
		The statements (1) and (2) are standard.
		
		We give the proofs for (3) and (4) in order to familiarize ourselves with the definitions.
		
		Regarding the base ordering, we observe that $(p,q)\leq(p',q)$ if and only if $p\leq p'$ in the ordering on $\Add(\tau,\gamma)$. So the base ordering is equal to $\Add(\tau,\gamma)$ which is $(2^{<\tau})^+=\tau^+$-Knaster.
		
		To show (4), let $(p,q_{\alpha})_{\alpha<\delta}$ (for $\delta<\mu$) be a descending sequence. Let $x:=\bigcup_{\alpha<\delta}\dom(q_{\alpha})$. $|x|<\mu$. On $x$, define a function $q$ as follows: Given $\gamma\in x$, let $\alpha_0$ be such that $\gamma\in\dom(q_{\alpha_0})$ (so $\gamma\in\dom(q_{\alpha})$ for all $\alpha\geq\alpha_0)$. Then $(q_{\alpha}(\gamma))_{\alpha_0\leq\alpha<\delta}$ is forced by $p$ to be a descending sequence in $\dot{\Coll}(\check{\mu},\check{\delta})$, so by the maximal principle we can fix $q(\gamma)$ which is forced to be a lower bound. It follows that $(p,q)$ is a lower bound of $(p,q_{\alpha})_{\alpha<\delta}$.
		
		Regarding the projection property, given $(p',q')\leq(p,q)$, let $q''$ be a function with domain $\dom(q')$ such that for all $\alpha\in\dom(q')$, $q''(\alpha)$ is forced by $p'$ to be equal to $q'(\alpha)$ and by conditions incompatible to $p'$ to be equal to $q(\alpha)$ (or $\emptyset$, if $q(\alpha)$ is not defined). Clearly, $q''$ is as required.
		
		$\dM_0$ has the refinement property since stronger conditions force more.
		
		For the mixing property, let $(p,q_0),(p,q_1)\leq (p,q)$. Let $p_0,p_1$ be extensions of $p$ such that $p_0(0)$ and $p_0(1)$ are incompatible. Construct a function $q'$ with domain $x:=\dom(q_0)\cup\dom(q_1)$ as follows: Let $\alpha\in x$. Then $p_0\uhr\alpha$ and $p_1\uhr\alpha$ are incompatible since $\alpha$ is the successor of a cardinal. So we can let $q'(\alpha)$ be such that it is forced by $p_0\uhr\alpha$ to be equal to $q_0(\alpha)$ (or $\emptyset$, if $q_0(\alpha)$ is not defined) and by conditions incompatible with $p_0\uhr\alpha$ (in particular, $p_1\uhr\alpha$) to be equal to $q_1(\alpha)$ (or $\emptyset$, if $q_1(\alpha)$ is not defined). It follows that $(p,q')\leq(p,q)$ and for $i=0,1$, $(p_i,q')\leq(p,q_i)$.
	\end{proof}
	
	As is common when working with variants of Mitchell Forcing, we will explicitly compute a version of the quotient order:
	
	\begin{mydef}
		Let $\nu\in(\mu,\kappa)$ be inaccessible. Let $G$ be $\dM_0(\tau,\mu,\nu)$-generic. In $V[G]$, let $\dM_0(G,\tau,\mu,\kappa\smallsetminus\nu)$ consist of pairs $(p,q)$ such that
		\begin{enumerate}
			\item $p\in\Add(\tau,\kappa\smallsetminus\nu)$
			\item $q$ is a ${<}\,\mu$-sized partial function on $\kappa\smallsetminus\nu$ such that for all $\gamma\in\dom(q)$, $\gamma=\delta+1$ for an inaccessible cardinal $\nu$ and $q(\gamma)$ is an $\Add(\tau,\gamma\smallsetminus\nu)$-name for a condition in $\dot{\Coll}(\check{\mu},\check{\nu})$.
		\end{enumerate}
		We let $(p',q')\leq(p,q)$ if
		\begin{enumerate}
			\item $p'\leq p$
			\item $\dom(q')\supseteq\dom(q)$ and for all $\gamma\in\dom(q)$
			$$p'\uhr(\gamma\smallsetminus\nu)\Vdash q'(\gamma)\leq q(\gamma)$$
		\end{enumerate}
	\end{mydef}
	
	Again, $\dM_0(G,\tau,\mu,\kappa\smallsetminus\nu)$ has very similar properties to $\dM_0(\tau,\mu,\kappa)$. In particular, because we are not immediately collapsing $\nu$ but adding a subset of $\tau$ first, we obtain that the order also has the mixing property.
	
	\begin{mylem}\label{M0QuotientProp}
		Let $\nu\in(\mu,\kappa)$ be inaccessible and let $G$ be $\dM_0(\tau,\mu,\nu)$-generic. In $V[G]$, the following holds:
		\begin{enumerate}
			\item $\dM_0(G,\tau,\mu,\kappa\smallsetminus\nu)$ is ${<}\,\tau$-directed closed.
			\item $\dM_0(G,\tau,\mu,\kappa\smallsetminus\nu)$ is $\kappa$-Knaster,
			\item The base ordering on $\dM_0(G,\tau,\mu,\kappa\smallsetminus\nu)$ is $\tau^+$-Knaster,
			\item The term ordering on $\dM_0(G,\tau,\mu,\kappa\smallsetminus\nu)$ is ${<}\,\mu$-closed,
			\item The ordering is iteration-like.
		\end{enumerate}
	\end{mylem}
	
	\begin{proof}
		This is entirely similar to Lemma \ref{M0Prop}. To show the mixing property, choose $p_0$ and $p_1$ such that $p_0(\nu)$ and $p_1(\nu)$ are incompatible (using that any $q$ is only defined on a subset of $(\nu,\kappa)$)
	\end{proof}
	
	Applying Theorem \ref{ApproxProp}, we obtain:
	
	\begin{mycol}\label{M0Approx}
		Let $\nu\in(\mu,\kappa)$ be inaccessible and let $G$ be $\dM_0(\tau,\mu,\nu)$-generic. In $V[G]$, the forcing $\dM_0(G,\tau,\mu,\kappa\smallsetminus\nu)$ has the ${<}\,\tau^+$-approximation property.
	\end{mycol}
	
	$\dM_0(G,\tau,\mu,\kappa\smallsetminus\nu)$ actually brings us from $V[G]$ to an extension by $\dM_0(\tau,\mu,\kappa)$:
	
	\begin{mylem}\label{M0Embedding}
		Let $\nu\in(\mu,\kappa]$ be inaccessible. There is a dense embedding from $\dM_0(\tau,\mu,\kappa)$ into $\dM_0(\tau,\mu,\nu)*\dM_0(\Gamma,\tau,\mu,\kappa\smallsetminus\nu)$.
	\end{mylem}
	
	\begin{proof}
		As is common, we define $\iota((p,q)):=((p\uhr\nu,q\uhr\nu),(\check{p}\uhr[\nu,\kappa),\overline{q}))$, where $\overline{q}$ is $\dM_0(\tau,\mu,\nu)$-name for a function on $\kappa\smallsetminus\nu$ given by recomputing $q\uhr[\nu,\kappa)$.
		
		It is clear that $\iota$ preserves $\leq$. Let $((p_0,q_0),\sigma)$ be a condition in $\dM_0(\tau,\mu,\nu)*\dM_0(\Gamma,\tau,\mu,\kappa\smallsetminus\nu)$. By strengthening $(p_0,q_0)$ if necessary, we can assume that it forces $\sigma=(\check{p_1},\sigma')$ for some $p_1\in\Add(\tau,\kappa\smallsetminus\nu)$ (because $\dM_0(\tau,\mu,\nu)$ is ${<}\,\tau$-directed closed). $\sigma'$ is forced to be a ${<}\,\mu$-sized sequence of ${<}\,\mu$-sized subsets of $V$, so we can (again, strengthening $(p_0,q_0)$ if necessary) assume that $\sigma'$ is an $\Add(\tau,\nu)$-name. Lastly, by the $\mu$-cc. of $\Add(\tau,\nu)$, there is a set $x\in V$ such that $(p_0,q_0)\Vdash\dom(\sigma')\subseteq\check{x}$. Now we can define a function $q_1$ on $x$ such that for $\gamma\in x$, $q_1(\gamma)$ is the $\Add(\tau,\gamma)$-name corresponding to $\sigma'(\check{\gamma})$, which is an $\Add(\tau,\nu)$-name for an $\Add(\tau,\gamma\smallsetminus\nu)$-name.
		
		It follows that $(p_0\cup p_1,q_0\cup q_1)$ has the property that $\iota(p_0\cup p_1,q_0\cup q_1)\leq((p_0,q_0),\sigma)$.
	\end{proof}
	
	We analyze the cardinals in the extension by $\dM_0(\tau,\mu,\kappa)$
	
	\begin{mylem}
		$\dM_0(\tau,\mu,\kappa)$ preserves cardinals below and including $\mu$, above and including $\kappa$ and forces $2^{\tau}=\kappa=\mu^+$
	\end{mylem}
	
	\begin{proof}
		By a nice name argument, $\dM_0(\tau,\mu,\kappa)$ forces $2^{\tau}=\kappa$: Any $\dM_0(\tau,\mu,\kappa)$-name for a subset of $\tau$ be can viewed as function from $\tau$ into the set of maximal antichains in $\dM_0(\tau,\mu,\kappa)$. Since $\dM_0(\tau,\mu,\kappa)$ is $\kappa$-cc and has size $\kappa$, that set has size $\kappa^{<\kappa}=\kappa$ by the inaccessibility of $\kappa$. By the same argument, there are at most $\kappa$ functions from $\tau$ into the set of maximal antichains in $\dM_0(\tau,\mu,\kappa)$ and so there are (up to equivalence) at most $\kappa$ many $\dM_0(\tau,\mu,\kappa)$-names for subsets of $\tau$. However, $\dM_0(\tau,\mu,\kappa)$ clearly projects onto $\Add(\tau,\kappa)$ and therefore forces $2^{\tau}\geq\kappa$.
		
		Because $\dM_0(\tau,\mu,\kappa)$ is ${<}\,\tau$-directed closed, it preserves cardinals below and including $\tau$. Because we can project onto it from the product of a $\tau^+$-cc. and a ${<}\,\mu$-closed poset, it preserves cardinals in the interval $[\tau^+,\mu]$. Lastly, every cardinal in $(\mu,\kappa)$ is collapsed to $\mu$ and cardinals above and including $\kappa$ are preserved by the $\kappa$-cc.
	\end{proof}
	
	The last thing left to show is
	
	\begin{mysen}\label{M0ISP}
		Let $\lambda\geq\kappa$ be a cardinal such that $\lambda^{<\kappa}=\lambda $ and $\kappa$ is $\lambda$-ineffable. $\dM_0(\tau,\mu,\kappa)$ forces $\ISP(\tau^+,\kappa,\lambda)$.
	\end{mysen}

	\begin{proof}
		Clearly $\dM_0(\tau,\mu,\kappa)$ is of size $\kappa$ and $\kappa$-cc. Let $e$ be a $(\kappa,\lambda)$-list. Let $\Theta$ be large. Choose any $\lambda$-ineffability witness $M$ for $\kappa$ with respect to $e$ as in Lemma \ref{IneffEmbed}. If $G$ is $\dM_0(\tau,\mu,\kappa)$-generic over $V$, $\pi[G\cap M]$ is $\pi(\dM_0(\tau,\mu,\kappa))=\dM_0(\tau,\mu,M\cap\kappa)$-generic over $V$. By Lemma \ref{M0Embedding}, $V[G]$ is an extension of $V[\pi[G\cap M]]$ by forcing with $\dM_0(G,\tau,\mu,\kappa\smallsetminus(M\cap\kappa))$, which has the ${<}\,\tau^+$-approximation property by Corollary \ref{M0Approx}. So by Theorem \ref{MainTheorem}, $\dM_0(\tau,\mu,\kappa)$ forces $\ISP(\tau^+,\kappa,\lambda)$.
	\end{proof}

	We can show that we have exactly $\ISP(\tau^+,\kappa,\lambda)$ (regarding the slenderness of lists):

	\begin{mylem}\label{NegISPTau}
		$\dM_0(\tau,\mu,\kappa)$ forces that $\ISP(\tau,\kappa,\kappa)$ fails.
	\end{mylem}
	
	We prove a more general statement:
	
	\begin{mylem}
		Assume $\delta<\theta$ are regular cardinals such that $2^{<\delta}<\theta$ and $2^{\delta}\geq\theta$. Then $\ISP(\delta,\theta,\theta)$ fails.
	\end{mylem}
	
	\begin{proof}
		Let $(x_{\alpha})_{\alpha<\theta}$ enumerate different subsets of $\delta$. Let $e$ be the following $(\theta,\theta)$-list: $e(a)=x_a$ if $a\geq\delta$ is an ordinal and $\emptyset$ otherwise.
		
		\begin{myclaim}
			$e$ is $\delta$-slender.
		\end{myclaim}
		
		\begin{proof}
			Let $\Theta$ be large and $C$ be the club of all $M\in[H(\Theta)]^{<\theta}$ such that $[\delta]^{<\delta}\subseteq M$ (here we use $2^{<\delta}<\theta$). If $M\in C$ and $x\in[\theta]^{<\delta}$, $e(M\cap\theta)\cap x\in[\delta]^{<\delta}$ (since $e(M\cap\theta)$ is either a subset of $\delta$ or empty) and thus in $M$.
		\end{proof}
		
		\begin{myclaim}
			$e$ does not have an ineffable branch.
		\end{myclaim}
		
		\begin{proof}
			Otherwise there is $b$ such that $S:=\{a\in[\theta]^{<\theta}\;|\;e(a)=b\cap a\}$ is stationary in $[\theta]^{<\theta}$. In particular $S\cap\theta$ is stationary in $\theta$. However, since $e(a)$ is a subset of $\delta$ for every $a$, $e$ is constant on $(S\cap\theta)\smallsetminus\delta$ which is an obvious contradiction.
		\end{proof}
		
		Thus we have produced a $\delta$-slender $(\theta,\theta)$-list without an ineffable branch.
	\end{proof}
	
	Lemma \ref{NegISPTau} follows because $\tau^{<\tau}=\tau<\kappa$ in $V$, which is preserved into $V[G]$ by the ${<}\,\tau$-directed closure of $\dM_0(\tau,\mu,\kappa)$. We obtain another Corollary:
	
	\begin{mycol}
		Assume $\kappa$ is not a strong limit and $\ISP(\delta,\kappa,\kappa)$ holds. Then $2^{<\delta}\geq\kappa$.
	\end{mycol}
	
	\begin{proof}
		Let $\mu<\kappa$ be minimal such that $2^{\mu}\geq\kappa$. If $\mu<\delta$ we are done so assume $\mu\geq\delta$. Then $2^{<\mu}<\kappa$ and $2^{\mu}\geq\kappa$, hence $\ISP(\mu,\kappa,\kappa)$ fails. But $\mu\geq\delta$, so $\ISP(\delta,\kappa,\kappa)$ fails as well since every $\mu$-slender list is also $\delta$-slender.
	\end{proof}

	The previous Corollary means that we must always pay a certain price to obtain strong versions of $\ISP$ (namely, blowing up $2^{<\delta}$).

	\section{$\ISP$ and Cardinal Arithmetic}

	We give easier constructions of two results which were known at $\omega_2$: We show that it is consistent that $\ISP(\tau^{++},\tau^{++},\lambda)$ holds and $2^{\tau}\neq 2^{\tau^+}$ and that $\ISP(\tau^+,\tau^{++},\lambda)$ is consistent together with an arbitrarily high value of $2^{\tau}$ (this was previously shown for $\tau=\omega$ in \cite{CoxKruegerStronglyProper}). In the model for the first statement, we also have that $\ISP(\tau^+,\tau^{++},\tau^{++})$ fails, answering a question of Weiss (which was previously answered in \cite{LambieStrongTreeKurepaGuessMod} for $\tau=\omega$).
	
	\begin{mydef}\label{DefM1}
		Let $\tau<\mu<\nu$ be regular cardinals such that $\tau^{<\tau}=\tau$, $\mu^{<\mu}=\mu$ and $\nu$ is inaccessible. For any ordinal $\gamma$, define $\dM_1(\tau,\mu,\nu,\gamma):=\dM_0(\tau,\mu,\nu)\times\Add(\mu,\gamma)$.
	\end{mydef}
	
	For the rest of this section, fix regular cardinals $\tau<\mu<\kappa\leq\lambda$ such that $\tau^{<\tau}=\tau$, $\mu^{<\mu}=\mu$ and $\kappa$ is inaccessible. Also use the same definition for $\Add(\tau,\gamma)$ as in the previous section.
	
	For $\dM_1$, we can show directly that the ``quotient ordering`` has the correct approximation property (of course building on the results for $\dM_0$).
	
	\begin{mylem}\label{LemmaApprox1}
		Let $\nu\in(\mu,\kappa]$ be an inaccessible cardinal and $\gamma$ an ordinal. Let $G'\times H'$ be $\dM_1(\tau,\mu,\nu,\gamma)$-generic. In $V[G'\times H']$, $\dM_0(G',\tau,\mu,\kappa\smallsetminus\nu)\times\Add(\mu,\lambda\smallsetminus \gamma)$ has the ${<}\,\nu$-approximation property.
	\end{mylem}
	
	\begin{proof}
		$\Add(\mu,\lambda\smallsetminus\gamma)$ is $\nu$-Knaster in $V[G'\times H']$ and thus has the ${<}\,\nu$-approximation property. Let $H''$ be $\Add(\mu,\lambda\smallsetminus \gamma)$-generic over $V[G'\times H']$. Then $V[G'\times H'][H'']$ is equal to $V[G'][H]$, where $H$ is $\Add(\mu,\lambda)^V$-generic over $V[G']$. We note that $\Add(\mu,\lambda)^V$ is ${<}\,\mu$-distributive in $V[G']$: By Lemma \ref{M0Prop} and Lemma \ref{PPImpliesProjection} $V[G']$ is contained in an extension by the product of a ${<}\,\mu$-closed and a $\tau^+$-Knaster forcing. In that extension, $\Add(\mu,\lambda)^V$ is clearly ${<}\,\mu$-distributive by Easton's Lemma and thus the same statement holds in the smaller model $V[G']$.
		
		In $V[G']$, $\dM_0(G',\tau,\mu,\kappa\smallsetminus\nu)$ is iteration-like and has $\Add(\tau,\kappa\smallsetminus\nu)$ as its base ordering as well as a ${<}\,\mu$-closed term ordering (by Lemma \ref{M0QuotientProp} and its proof). Thus, in $V[G'][H]$, the ordering is still iteration-like (as this property is absolute), the base ordering is still $\tau^+$-Knaster and the term ordering is still ${<}\,\mu$-closed. Hence $\dM_0(G',\tau,\mu,\kappa\smallsetminus\nu)$ has the ${<}\,\tau^+$-approximation property in $V[G'][H]=V[G'\times H'][H'']$.
		
		Now assume there is $f\in V[G\times H]$ such that $f\cap z\in V[G'\times H']$ for every $z\in[V[G'\times H']]^{<\nu}\cap V[G'\times H']$. Let $z\in V[G'\times H]$ have size ${<}\,\nu$. Because $\Add(\mu,\lambda)$ is $\nu$-Knaster there is $y\in V[G'\times H']$ with $z\subseteq y$ and $|y|<\nu$. Hence $f\cap y\in V[G'\times H']$ and $f\cap z=(f\cap y)\cap z\in V[G'\times H]$. As $z$ was arbitrary and $\dM_0(G',\tau,\mu,\kappa\smallsetminus\nu)$ has the ${<}\,\tau^+$-approximation property in $V[G'\times H']$, $f\in V[G'\times H]$. Now since $\Add(\mu,\lambda\smallsetminus\gamma)$ has the ${<}\,\nu$-approximation property in $V[G'\times H']$, $f\in V[G'\times H']$.
	\end{proof}
	
	A straightforward application of Theorem \ref{MainTheorem} shows:
	
	\begin{mysen}
		Assume $\lambda_0\geq\lambda$ is a regular cardinal with $\lambda_0^{<\kappa}=\lambda_0$ such that $\kappa$ is $\lambda_0$-ineffable. After forcing with $\dM_1(\tau,\mu,\kappa,\lambda)$, $2^{\tau}=\kappa, 2^{\mu}=\lambda$ and $\ISP(\kappa,\kappa,\lambda_0)$ hold.
	\end{mysen}
	
	\begin{proof}
		Again, $\dM_1(\tau,\mu,\kappa,\lambda)$ has size $\leq\lambda_0$ and is $\kappa$-cc since it is the product of two $\kappa$-Knaster posets (this is clear for $\Add(\mu,\lambda)$ and was shown in Lemma \ref{M0Prop} for $\dM_0(\tau,\mu,\kappa)$). Let $e$ be any $(\kappa,\lambda)$-list, $\Theta$ large and $x\in[H(\Theta)]^{<\kappa}$. Choose any $\lambda$-ineffability witness $M$ for $\kappa$ with respect to $e$ as in Lemma \ref{IneffEmbed} with $\tau^+\in M$. We have $\pi(\dM_1(\tau,\mu,\kappa,\lambda))=\dM_0(\tau,\mu,M\cap\kappa)\times\Add(\mu,\pi(\lambda))$. If $G$ is $\dM_1(\tau,\mu,\kappa,\lambda)$-generic, then $\pi[G\cap M]$ is $\dM_0(\tau,\mu,M\cap\kappa)\times\Add(\mu,\pi(\lambda))$-generic over $V$ by its $M\cap\kappa$-cc (again, it is the product of two $M\cap\kappa$-Knaster posets). $V[G]$ is an extension of $V[\pi[G\cap M]]$ by $\dM_0(G',\tau,\mu,\kappa\smallsetminus(M\cap\kappa))\times\Add(\mu,\lambda\smallsetminus\pi(\lambda))$ which has the ${<}\,M\cap\kappa$-approximation property in $V[\pi[G\cap M]]$ by Lemma \ref{LemmaApprox1}. So we can once again apply Theorem \ref{MainTheorem}.
	\end{proof}
	
	We also obtain an answer to another question of Mohammadpour which was previously answered in \cite{LambieStrongTreeKurepaGuessMod} in the case $\omega$: It is consistent that $\ISP(\kappa,\kappa,\lambda_0)$ holds (with $\kappa=\mu^+$) but $\ISP(\mu,\kappa,\kappa)$ fails:
	
	\begin{mylem}
		After forcing with $\dM_1(\tau,\mu,\kappa,\lambda)$, $\ISP(\mu,\kappa,\kappa)$ fails.
	\end{mylem}
	
	\begin{proof}
		As before, let $f(a)$ be the $a$th Cohen subset of $\mu$ added by $\dM_1(\tau,\mu,\kappa,\lambda)$ if $a\geq\mu$ is an ordinal and $\emptyset$ otherwise. Because every ${<}\,\mu$-sized segment of $f(a)$ is in $([\mu]^{<\mu})^{V}$ which has size $\mu<\kappa$, we see that $f$ is $\mu$-slender. However, by previous arguments $f$ cannot have an ineffable branch.
	\end{proof}
	
	Now we show that $\ISP(\tau^+,\kappa,\lambda)$ is consistent with an arbitrarily large continuum.
	
	\begin{mydef}
		Let $\tau<\mu<\nu$ be regular cardinals such that $\tau^{<\tau}=\tau$ and $\nu$ is inaccessible. For any ordinal $\gamma$, define $\dM_2(\tau,\mu,\nu,\gamma):=\dM_0(\tau,\mu,\nu)\times\Add(\tau,\gamma)$.
	\end{mydef}
	
	For the rest of this section, we drop the assumption that $\mu^{<\mu}=\mu$.
	
	We have a very similar Lemma to before (albeit with a stronger approximation property):
	
	\begin{mylem}
		Let $\nu\in(\mu,\kappa)$ be inaccessible and let $G'\times H'$ be $\dM_2(\tau,\mu,\nu,\gamma)$-generic. In $V[G'\times H']$, $\dM_0(G',\tau,\mu,\kappa\smallsetminus\nu)\times\Add(\tau,\lambda\smallsetminus\gamma)$ has the ${<}\,\tau^+$-approximation property.
	\end{mylem}
	
	\begin{proof}
		In $V[G'\times H']$, $\Add(\tau,\lambda\smallsetminus \gamma)$ is still $(2^{<\tau})^{V[G'\times H']}=(\tau^+)^{V[G'\times H']}=(\tau^+)^V$-cc., so $\Add(\tau,\lambda\smallsetminus \gamma)$ has the ${<}\,\tau^+$-approximation property in $V[G'\times H']$. Let $H''$ be $\Add(\tau,\lambda\smallsetminus \gamma)$ be $\Add(\tau,\lambda\smallsetminus \gamma)$-generic and $H$ the $\Add(\tau,\lambda)$-generic filter induced by $H'$ and $H''$. In $V[G']$, $\dM_0(G',\tau,\tau^+,\kappa\smallsetminus\nu)$ has a ${<}\,\tau^+$-Knaster base ordering and a ${<}\,\tau^+$-closed term ordering. In $V[G'\times H]$, the base ordering is still ${<}\,\tau^+$-Knaster and the term ordering is at least ${<}\,\tau^+$-strongly distributive, because $V[G'\times H]$ is an extension of $V[G']$ by a $\tau^+$-Knaster forcing. Being iteration-like is absolute and thus $\dM_0(G',\tau,\tau^+,\kappa\smallsetminus\nu)$ has the ${<}\,\tau^+$-approximation property in $V[G'\times H]$. Now proceed as in Lemma \ref{LemmaApprox1}.
	\end{proof}
	
	And we can prove:
	
	\begin{mysen}
		Let $\tau<\kappa\leq\lambda=\lambda^{<\kappa}\leq\lambda_0$ be cardinals such that $\tau^{<\tau}=\tau$ and $\kappa$ is $\lambda_0$-ineffable. After forcing with $\dM_2(\tau,\kappa,\lambda)$, $2^{\tau}=\lambda$ and $\ISP(\tau^+,\kappa,\lambda_0)$ hold.
	\end{mysen}
	
	\begin{proof}
		This follows just as for $\dM_1$.
	\end{proof}
	
	Interestingly, we needed a certain degree of ineffability to obtain a large powerset of $\mu$ (or $\tau$ respectively). This begs the following question:
	
	\begin{myque}
		Assume $\ISP(\delta,\kappa,\kappa)$ holds and $\kappa^{<\delta}=\lambda$. Does this imply $\ISP(\delta,\kappa,\lambda)$? More generally, does $\ISP(\delta,\kappa,\lambda)$ imply $\ISP(\delta,\kappa,\lambda^{<\delta})$?
	\end{myque}
	
	\section{Indestructibility of $\ISP$}
	
	It is a well-known result by Laver that if $\kappa$ is a supercompact cardinal, there is a forcing which leaves $\kappa$ supercompact and moreover makes the supercompactness of $\kappa$ indestructible under ${<}\,\kappa$-directed closed forcing (see \cite{LaverIndestruct}). It is also known that if the \emph{proper forcing axiom} $\PFA$ holds, the tree property at $\aleph_2$ is indestructible by ${<}\,\aleph_2$-closed forcing (see \cite{KoenigYoshinobuFragmentsMM}). In \cite{UngerFragIndestructI}, Unger devised a guessing variant of Mitchell forcing to present a more adaptable method for obtaining the indestructibility of the tree property under directed-closed forcings (in particular, this method can be carried out at cardinals above $\aleph_2$). We will adapt his arguments to show that it is consistent that $\ISP(\tau^+,\kappa,\geq\kappa)$ holds (for $\kappa$ a successor of a regular cardinal), where $\ISP(\tau^+,\kappa,\geq\kappa)$ means that $\ISP(\tau^+,\kappa,\lambda)$ holds for every $\lambda\geq\kappa$, and is indestructible under ${<}\,\kappa$-directed closed forcing.
	
	Our forcing will be a guessing variant of Mitchell forcing, modified to collapse cardinals in a ``non-fresh`` way to ensure the approximation property.
	
	We will need the concept of a Laver function:
	
	\begin{mydef}
		Let $\kappa$ be a cardinal and $l\colon \kappa\to V_{\kappa}$. $l$ is a \emph{Laver function} if for any $A$ and any $\lambda\geq|\tcl(A)|$ there is an embedding $j\colon V\to M$ with critical point $\kappa$ such that $j(\kappa)>\lambda$, $j(f)(\kappa)=A$ and $M^{\lambda}\subseteq M$.
	\end{mydef}
	
	Similar functions can be defined for many large cardinal notions witnessed by embeddings. In the author's PhD thesis, such functions are defined and constructed (by forcing) for $\lambda$-ineffable cardinals. However, in the case of supercompactness, a Laver function always exists:
	
	\begin{mysen}[\cite{LaverIndestruct}]
		Let $\kappa$ be a supercompact cardinal. Then there exists a Laver function $l\colon \kappa\to V_{\kappa}$.
	\end{mysen}
	
	Such a function also has interesting properties regarding the embeddings witnessing $\lambda$-ineffability.
	
	\begin{mylem}\label{LaverIneffable}
		Let $\kappa$ be a supercompact cardinal and let $l\colon\kappa\to V_{\kappa}$ be a Laver function. Then the following holds: For any $\lambda\geq\kappa$, any $(\kappa,\lambda)$-list $f$, any $A\in H(\lambda^+)$ and any large enough $\Theta$ there exists a $\lambda$-ineffability witness $N\prec H(\Theta)$ for $\kappa$ with respect to $f$ such that, letting $\pi$ be the Mostowski-Collapse of $N$, $l(\pi(\kappa))=\pi(A)$.
	\end{mylem}
	
	\begin{proof}
		Let $f$ be a $(\kappa,\lambda)$-list. Let $A\in H(\lambda^+)$, $\Theta$ be large and $x\in [H(\Theta)]^{<\kappa}$. Let $\delta:=|H(\Theta)|$ and $j\colon V\to M$ an elementary embedding such that $j(\kappa)>\delta$, ${}^{\delta}M\subseteq M$ and $j(l)(\kappa)=A$. By the closure properties $j[\lambda]\in M$, so $b:=j^{-1}[j(f)(j[\lambda])]$ is defined. The set $C$ of all $N\prec H(\Theta)$ of size ${<}\,\kappa$ containing $b,f,\kappa,\lambda,A$ is a club. Hence $N:=j[H(\Theta)]$ is in $j(C)$. We have
		$$j(b)\cap (N\cap j(\lambda))=j(b)\cap j[\lambda]=j(f)(j[\lambda])$$
		and $N$ is a $j(\lambda)$-ineffability witness for $j(\kappa)$ with respect to $j(f)$.
		
		\begin{myclaim}
			If $x\in H(\Theta)$, $\pi(j(x))=x$.
		\end{myclaim}
		
		\begin{proof}
			Note first that the statement makes sense because if $x\in H(\Theta)$, $j(x)\in N$.
			
			We prove it by $\in$-induction. Assume $y\in H(\Theta)$ and $\pi(j(x))=x$ for every $x\in y$. By the definition
			$$\pi(j(y))=\{\pi(a)\;|\;a\in j(y)\cap N\}$$
			Assume $a\in j(y)\cap N$. Because $j(y)\in j(H(\Theta))$, $a\in N\cap j(H(\Theta))=j[H(\Theta)]$, so there is $x\in H(\Theta)$ with $j(x)=a$. By assumption $x\in y$ so $\pi(j(x))=x\in\pi(j(y))$.
			
			Assume $x\in y$. Then $j(x)\in j(y)$ and $j(x)\in N$. Thus $\pi(j(x))=x\in\pi(j(y))$.
		\end{proof}
		
		Thus $j(l)(\pi(j(\kappa)))=j(l)(\kappa)=A=\pi(j(A))$. In summary, in $M$, there exists a $j(\lambda)$-ineffability witness (namely $N$) for $j(\kappa)$ with respect to $j(f)$ such that $j(l)(\pi(j(\kappa)))=\pi(j(A))$.
		
		By elementarity there is in $V$ a $\lambda$-ineffability witness for $\kappa$ with respect to $f$ such that $l(\pi(\kappa))=\pi(A)$.
	\end{proof}
	
	The previous lemma allows us to use small embeddings to prove our last result.
	
	For technical reasons, let $\Add(\tau,\beta)$ consist of functions $f$ on $\beta$ of size ${<}\,\tau$ such that for all $\gamma\in\dom(f)$, $\gamma$ is a successor ordinal and $f(\gamma)\in\Add(\tau)$.
	
	\begin{mydef}
		Let $\tau<\mu<\kappa$ be regular cardinals such that $\tau^{<\tau}=\tau$ and $\kappa$ is inaccessible. Let $l\colon\kappa\to V_{\kappa}$ be any function. We define $\dM_3^l(\tau,\mu,\beta)$ by induction on $\beta$, letting $\dM_3^l(\tau,\mu,0):=\{0\}$. Assume $\dM_3^l(\tau,\mu,\gamma)$ has been defined for all $\gamma<\beta$. Let $\dM_3^l(\tau,\mu,\beta)$ consist of triples $(p,q,r)$ such that
		\begin{enumerate}
			\item $p\in\Add(\tau,\beta)$
			\item $q$ is a partial function on $\beta$ of size ${<}\,\mu$ such that if $\gamma\in\dom(q)$, $\gamma=\delta+2$ for an inaccessible cardinal $\delta$ and $q(\gamma)$ is an $\Add(\tau,\gamma)$-name for a condition in $\dot{\Coll}(\check{\mu},\check{\delta})$
			\item $r$ is a partial function on $\beta$ of size ${<}\,\mu$ such that if $\gamma\in\dom(r)$, $\gamma$ is inaccessible, $l(\gamma)$ is an $\dM_3^l(\tau,\mu,\gamma)$-name for a ${<}\,\gamma$-directed closed partial order and $r(\gamma)$ is an $\dM_3^l(\tau,\mu,\gamma)$-name for an element of $l(\gamma)$.
		\end{enumerate}
		We let $(p',q',r')\leq(p,q,r)$ if
		\begin{enumerate}
			\item $p'\leq p$
			\item $\dom(q')\supseteq\dom(q)$ and for $\gamma\in\dom(q)$,
			$$p'\uhr\gamma\Vdash q'(\gamma)\leq q(\gamma)$$
			\item $\dom(r')\supseteq\dom(r)$ and for $\gamma\in\dom(r)$,
			$$(p'\uhr\gamma,q'\uhr\gamma,r'\uhr\gamma)\Vdash r'(\gamma)\leq r(\gamma)$$
		\end{enumerate}
	\end{mydef}
	
	We can view $\dM_3^l(\tau,\mu,\nu)$ as an order on a product in two distinct ways: Letting $\dP:=\Add(\tau,\nu)$, $\dQ$ the set of all possible $q$ and $\dR$ the set of all possible $r$, $\dM_3^l(\tau,\mu,\nu)$ is a based ordering on $\dP\times(\dQ\times\dR)$ and $(\dP\times\dQ)\times\dR$. We will choose the first option. We also note that the base ordering induced on $(\dP\times\dQ)$ by $(\dP\times\dQ)\times\dR$ is isomorphic to $\dM_0$.
	
	\begin{mylem}
		Let $\nu\in(\mu,\kappa]$ be inaccessible such that $l\uhr\nu\colon \nu\to V_{\nu}$.
		\begin{enumerate}
			\item $\dM_3^l(\tau,\mu,\nu)$ is ${<}\,\tau$-directed closed,
			\item $\dM_3^l(\tau,\mu,\nu)$ is $\nu$-Knaster,
			\item The base ordering on $\dM_3^l(\tau,\mu,\nu)$ is $\tau^+$-Knaster,
			\item The term ordering on $\dM_3^l(\tau,\mu,\nu)$ is ${<}\,\mu$-closed,
			\item The ordering is iteration-like.
		\end{enumerate}
	\end{mylem}
	
	\begin{proof}
		The proofs for (1), (2) and (3) are standard.
		
		Regarding (4), let $(p,q_{\alpha},r_{\alpha})_{\alpha<\delta}$ be a descending sequence (with $\delta<\mu$). We first define $q$ as in Lemma \ref{M0Prop}, so that $(p,q)$ is a lower bound of $(p,q_{\alpha})_{\alpha<\delta}$. Now let $y:=\bigcup_{\alpha<\delta}\dom(r_{\alpha})$, which has size ${<}\,\mu$. We define a function $r$ on $y$ by induction on $\beta<\nu$ such that $(p\uhr\beta,q\uhr\beta,r\uhr\beta)$ is a lower bound of $(p\uhr\beta,q\uhr\beta,r_{\alpha}\uhr\beta)_{\alpha<\delta}$. Assume $r\uhr\beta$ has been defined and $\beta\in y$. Let $\alpha_0$ be such that $\beta\in\dom(r_{\alpha_0})$. For $\alpha_0\leq\alpha<\alpha'<\delta$, $(p,q_{\alpha'},r_{\alpha'})\leq(p,q_{\alpha},r_{\alpha})$ and thus
		$$(p\uhr\beta,q_{\alpha'}\uhr\beta,r_{\alpha'}\uhr\beta)\Vdash r_{\alpha'}(\beta)\leq r_{\alpha}(\beta)$$
		By the inductive hypothesis,
		$$(p\uhr\beta,q\uhr\beta,r\uhr\beta)\leq(p\uhr\beta,q\uhr\beta,r_{\alpha'}\uhr\beta)\leq(p\uhr\beta,q_{\alpha'}\uhr\beta,r_{\alpha'}\uhr\beta)$$
		so $(p\uhr\beta,q\uhr\beta,r\uhr\beta)\Vdash r_{\alpha'}(\beta)\leq r_{\alpha}(\beta)$. Thus $(p\uhr\beta,q\uhr\beta,r\uhr\beta)$ forces $(r_{\alpha}(\beta))_{\alpha_0\leq\alpha<\delta}$ to be a descending sequence in some ${<}\,\beta$-directed closed partial order (where $\beta\geq\mu$) and we can fix a lower bound $r(\beta)$. Then $(p\uhr\beta+1,q\uhr\beta+1,r\uhr\beta+1)$ is a lower bound of $(p\uhr\beta+1,q_{\alpha}\uhr\beta+1,r_{\alpha}\uhr\beta+1)_{\alpha<\delta}$.
		
		Now we show (5). Regarding the projection property, let $(p',q',r')\leq(p,q,r)$. Find $q''$ as in Lemma \ref{M0Prop} such that $(p,q'')\leq(p,q)$ and $(p',q'')\leq(p',q')\leq(p',q'')$. We define a function $r''$ with domain $\dom(r')$ by induction such that for every $\beta$,
		$$(p\uhr\beta,q''\uhr\beta,r''\uhr\beta)\leq(p\uhr\beta,q\uhr\beta,r\uhr\beta)$$
		and
		$$(p'\uhr\beta,q''\uhr\beta,r''\uhr\beta)\leq(p'\uhr\beta,q'\uhr\beta,r'\uhr\beta)\leq(p'\uhr\beta,q''\uhr\beta,r''\uhr\beta)$$
		Assume $r''\uhr\beta$ has been defined and $\beta\in\dom(r')$. Let $r''(\beta)$ be a name such that $(p'\uhr\beta,q''\uhr\beta,r''\uhr\beta)$ forces $r''(\beta)=r'(\beta)$ and conditions incompatible with $(p'\uhr\beta,q''\uhr\beta,r''\uhr\beta)$ force $r''(\beta)=r(\beta)$. Then
		$$(p\uhr\beta+1,q''\uhr\beta+1,r''\uhr\beta+1)\leq(p\uhr\beta+1,q\uhr\beta+1,r\uhr\beta+1)$$
		and
		$$(p'\uhr\beta+1,q''\uhr\beta+1,r''\uhr\beta+1)\leq(p'\uhr\beta+1,q'\uhr\beta+1,r'\uhr\beta+1)\leq(p'\uhr\beta+1,q''\uhr\beta+1,r''\uhr\beta+1)$$
		
		The refinement property is clear.
		
		Regarding the mixing property, let $(p,q_0,r_0),(p,q_1,r_1)\leq(p,q,r)$. Choose $p_0,p_1\leq p$ such that $p_0(1)$ and $p_1(1)$ are incompatible and find $q'$ such that $(p,q')\leq(p,q)$ and $(p_i,q')\leq(p,q_i)$. It follows that for any inaccessible cardinal $\gamma$, $(p_0\uhr\gamma,q'\uhr\gamma)$ and $(p_1\uhr\gamma,q'\uhr\gamma)$ are incompatible. Define a function $r'$ on $\dom(r_0)\cup\dom(r_1)$ by induction on $\beta$ such that $(p\uhr\beta,q'\uhr\beta,r'\uhr\beta)\leq(p,q\uhr\beta,r\uhr\beta)$ and $(p_i\uhr\beta,q'\uhr\beta,r'\uhr\beta)\leq(p\uhr\beta,q_i\uhr\beta,r_i\uhr\beta)$. Assume $\beta\in\dom(r_0)\cup\dom(r_1)$ (if $\beta$ is only in one domain, set $r_i(\beta)=\emptyset$ for the other $i$). By the inductive hypothesis, $(p_i\uhr\beta,q'\uhr\beta,r'\uhr\beta)$ are incompatible for $i\in 2$, so we can choose a name $r'(\beta)$ that is forced by $(p_0\uhr\beta,q'\uhr\beta,r'\uhr\beta)$ to be equal to $r_0(\beta)$ and by conditions incompatible with $(p_0\uhr\beta,q'\uhr\beta,r'\uhr\beta)$ to be equal to $r_1(\beta)$. Then $(p\uhr\beta+1,q'\uhr\beta+1,r'\uhr\beta+1)\leq(p\uhr\beta+1,q\uhr\beta+1,r\uhr\beta+1)$ and $(p_i\uhr\beta+1,q'\uhr\beta+1,r'\uhr\beta+1)\leq(p\uhr\beta+1,q_i\uhr\beta+1,r_i\uhr\beta+1)$.
	\end{proof}
	
	As before, we explicitely construct the quotient ordering:
	
	\begin{mydef}
		Let $\nu<\kappa$ be an inaccessible cardinal such that $l\uhr\nu\colon\nu\to V_{\nu}$. Let $\xi:=\nu+1$ and $G$ be $\dM_3^l(\tau,\mu,\xi)$-generic. In $V[G]$, define the partial order $\dM_3^l(G,\tau,\mu,\kappa\smallsetminus\xi,\beta)$ by induction on $\beta\geq\xi$, letting $\dM_3^l(G,\tau,\mu,\kappa\smallsetminus\xi,\xi):=\{\emptyset\}$ (this defines a name $\dM_3^l(\Gamma,\tau,\mu,\kappa\smallsetminus\xi,\beta)$) Assume that for all $\gamma\in[\xi,\beta)$, $\dM_3^l(\Gamma,\tau,\mu,\kappa\smallsetminus\xi,\gamma)$ has been defined and there is a dense embedding from $\dM_3^l(\tau,\mu,\gamma)$ into $\dM_3^l(\tau,\mu,\xi)*\dM_3^l(\Gamma,\tau,\mu,\kappa\smallsetminus\xi,\gamma)$. Then we let $\dM_3^l(G,\tau,\mu,\kappa\smallsetminus\xi,\beta)$ consist of tripels $(p,q,r)$ such that
		\begin{enumerate}
			\item $p\in\Add(\tau,\beta\smallsetminus\xi)$
			\item $q$ is a ${<}\,\mu$-sized partial function on $\beta\smallsetminus\xi$ such that for all $\gamma\in\dom(q)$, $\gamma$ is the double successor of an inaccessible cardinal $\delta$ and $q(\gamma)$ is an $\Add(\tau,\gamma\smallsetminus\delta)$-name for an element in $\dot{\Coll}(\check{\mu},\check{\delta})$.
			\item $r$ is a ${<}\,\mu$-sized partial function on $\beta\smallsetminus\xi$ such that for all $\gamma\in\dom(r)$, $\gamma$ is inaccessible, $l(\gamma)$ is an $\dM_3^l(\tau,\mu,\gamma)$-name for a $<\gamma$-directed closed partial order and $r(\gamma)$ is an $\dM_3^l(G,\tau,\mu,\kappa\smallsetminus\xi,\gamma)$-name for an element in $l(\gamma)^G$ (reimagining $l(\gamma)$ as an $\dM(\tau,\mu,\xi)*\dM_3(\Gamma,\tau,\mu,\kappa\smallsetminus\xi,\gamma)$-name).
		\end{enumerate}
	\end{mydef}
	
	The following Lemma shows that the above construction proceeds up to $\kappa$ and actually defines a version of the quotient forcing:
	
	\begin{mylem}
		Let $\xi=\nu+1$ for an inaccessible cardinal $\nu$. For any $\beta\in[\xi,\kappa]$, there is a dense embedding from $\dM_3^l(\tau,\mu,\beta)$ into $\dM_3^l(\tau,\mu,\xi)*\dM_3^l(\Gamma,\tau,\mu,\kappa\smallsetminus\xi,\beta)$.
	\end{mylem}
	
	\begin{proof}
		We do the proof by induction on $\beta\in[\xi,\kappa]$ with the base case being clear. Assume that such an embedding exists for all $\gamma\in[\xi,\beta)$. So in particular, the $\dM_3^l(\Gamma,\tau,\mu,\kappa\smallsetminus\xi,\beta)$ is defined. We define the following embedding:
		$$\iota((p,q,r))=((p\uhr\xi,q\uhr\xi,r\uhr\xi),\op(\check{p\uhr[\xi,\beta)},\overline{q},\overline{r}))$$
		where $\overline{q}$ and $\overline{r}$ are defined as follows:
		
		$\overline{q}$ is an $\dM_3^l(\tau,\mu,\xi)$-name for a function such that $\Vdash``\dom(\overline{q})=\check{\dom(q)\smallsetminus\xi}''$ and for any $\gamma\in\dom(q)\smallsetminus\xi$, $\overline{q}(\check{\gamma})$ is an $\Add(\tau,\xi)$-name (and thus an $\dM_3^l(\tau,\mu,\xi)$-name) for an $\Add(\tau,\gamma\smallsetminus\xi)$-name corresponding to the $\Add(\tau,\gamma)$-name $q(\xi)$.
		
		$\overline{r}$ is an $\dM_3^l(\tau,\mu,\xi)$-name for a function such that $\Vdash\dom(\overline{r})=\check{\dom(r)\smallsetminus\xi}$ and for $\gamma\in\dom(r)\smallsetminus\xi$, $\overline{r}(\check{\gamma})$ is an $\dM_3^l(\tau,\mu,\xi)$-name for an $\dM_3^l(\tau,\mu,\gamma\smallsetminus\xi)$-name corresponding to the $\dM_3^l(\tau,\mu,\gamma)$-name $r(\gamma)$ (using the inductive hypothesis).
		
		It follows as in Lemma \ref{M0Embedding} that $\iota$ is a dense embedding.
	\end{proof}
	
	We let $\dM_3^l(G,\tau,\mu,\kappa\smallsetminus\xi):=\dM_3^l(G,\tau,\mu,\kappa\smallsetminus\xi,\kappa)$. This ordering has properties very similar to $\dM_3^l(\tau,\mu,\kappa)$:
	
	\begin{mylem}
		Let $\nu\in(\mu,\kappa]$ be inaccessible such that $l\uhr\nu\colon\nu\to V_{\nu}$ and $\xi:=\nu+1$. Let $G$ be $\dM_3^l(\tau,\mu,\xi)$-generic. In $V[G]$, the following holds:
		\begin{enumerate}
			\item $\dM_3^l(G,\tau,\mu,\kappa\smallsetminus\xi)$ is ${<}\,\tau$-directed closed,
			\item $\dM_3^l(G,\tau,\mu,\kappa\smallsetminus\xi)$ is $\kappa$-Knaster,
			\item The base ordering on $\dM_3^l(G,\tau,\mu,\kappa\smallsetminus\xi)$ is $\tau^+$-Knaster,
			\item The term ordering on $\dM_3^l(G,\tau,\mu,\kappa\smallsetminus\xi)$ is ${<}\,\mu$-closed,
			\item The ordering is iteration-like.
		\end{enumerate}
	\end{mylem}
	
	So in particular:
	
	\begin{mycol}
		Let $\xi=\nu+1$ for an inaccessible cardinal $\nu$. Let $G$ be $\dM_3^l(\tau,\mu,\xi)$-generic. In $V[G]$, let $\dot{\dL}$ be an $\dM_3^l(G,\tau,\mu,\kappa\smallsetminus\xi)$-name for a ${<}\,\kappa$-directed closed partial order. Then $\dM_3^l(G,\tau,\mu,\kappa\smallsetminus\xi)*\dot{\dL}$ has the ${<}\,\tau^+$-approximation property.
	\end{mycol}
	
	\begin{proof}
		We can view $\dM_3^l(G,\tau,\mu,\kappa\smallsetminus\xi)*\dot{\dL}$ as an ordering on a product $\dP\times\dQ$ by letting $\dP:=\Add(\tau,\kappa\smallsetminus\xi)$ and $\dQ$ consist of tripels $(q,r,\sigma)$ with $q,r$ as in $\dM_3^l(G,\tau,\mu,\kappa\smallsetminus\xi)$ and $\sigma$ such that $\Vdash\sigma\in\dot{\dL}$. This ordering is easily seen to be iteration-like. Moreover, the base ordering is $\tau^+$-Knaster (this is clear) and the term ordering is ${<}\,\mu$-closed: Let $(p,(q_{\alpha},r_{\alpha},\sigma_{\alpha}))_{\alpha<\delta}$ for $\delta<\mu$ be a descending sequence. Then the sequence $(p,(q_{\alpha},r_{\alpha}))_{\alpha<\delta}$ is a descending sequence in the term ordering on $\dM_3^l(G,\tau,\mu,\kappa\smallsetminus\xi)$. Therefore we can find a lower bound $(p,(q,r))$ of $(p,(q_{\alpha},r_{\alpha}))_{\alpha<\delta}$ in $\dM_3^l(G,\tau,\mu,\kappa\smallsetminus\xi)$. Then $(p,(q,r))$, being a lower bound, forces the sequence $(\sigma_{\alpha})_{\alpha<\delta}$ to be descending in the ${<}\,\kappa$-directed closed forcing $\dot{\dL}$. Hence by the maximal principle we can fix $\sigma$ which is forced to be a lower bound of $(\sigma_{\alpha})_{\alpha<\delta}$. In summary, $(p,(q,r,\sigma))$ is a lower bound of $(p,(q_{\alpha},r_{\alpha},\sigma_{\alpha}))_{\alpha<\delta}$.
		
		Now apply Theorem \ref{ApproxProp}.
	\end{proof}
	
	Now we can show that $\dM_3^l(\tau,\mu,\kappa)$ forces $\ISP$ in such a way that it is indestructible under ${<}\,\kappa$-directed closed forcing.
	
	\begin{mysen}
		Let $\tau<\mu<\kappa$ be regular cardinals such that $\tau^{<\tau}=\tau$ and $\kappa$ is supercompact. Let $l$ be a Laver function. Then $\dM_3^l(\tau,\mu,\kappa)$ forces $\ISP(\tau^+,\kappa,\geq\kappa)$ and it is indestructible under ${<}\,\kappa$-directed closed forcing.
	\end{mysen}
	
	\begin{proof}
		For any ordinal $\xi$, let $\dM(\xi):=\dM_3^l(\tau,\mu,\xi)$.
		
		It suffices to show that for any $\dM(\kappa)$-name $\dot{\dL}$ for a ${<}\,\kappa$-directed closed forcing, $\dM(\kappa)*\dot{\dL}$ forces $\ISP(\tau^+,\kappa,\geq\kappa)$. Additionally, it suffices to show that $\ISP(\tau^+,\kappa,\lambda)$ is forced for arbitrarily large $\lambda$.
		
		Thus let $\dot{\dL}$ be an $\dM(\kappa)$-name for a ${<}\,\kappa$-directed closed forcing, let $\lambda\geq|\dot{\dL}|$ and $\dot{f}$ an $\dM(\kappa)*\dot{\dL}$-name for a ${<}\,\tau^+$-slender $(\kappa,\lambda)$-list, forced by some $(p,\sigma)\in\dM(\kappa)*\dot{\dL}$. Let $\dot{F}$ be an $\dM(\kappa)*\dot{\dL}$-name for a function such that any $M\prec H(\Theta')$ closed under $\dot{F}$ witnesses the slenderness. We will view $\dM(\kappa)*\dot{\dL}$ as a partial order on $\lambda$. Let $\Theta$ be large and $M\prec H(\Theta)$ with the following:
		\begin{enumerate}
			\item $\nu:=M\cap\kappa\in\kappa$ is inaccessible
			\item $[M\cap\lambda]^{<\nu}\subseteq M$
			\item $\{\kappa,\lambda,\dM(\kappa),\dot{\dL},\dot{F},l\}\cup p\subseteq M$
			\item $l(\nu)=\pi(\dot{\dL})$.
		\end{enumerate}
		Then $p$ is actually a condition in $\dM(\nu)$ and $\pi(\sigma)$ is an $\dM(\nu)$-name for a condition in $\pi(\dot{\dL})$. Thus we can view $(p,\pi(\sigma))$ as a condition $q$ in $\dM(\nu+1)$. Let $G_0$ be an $\dM(\kappa)$-generic filter containing $q$ and work in $V[G_0]$.
		
		The collapse $\pi\colon M\to N$ extends to $M[G_0]\to N[\pi[G_0\cap M]]$ and, by the $\nu$-cc. of $\dM(\nu)$, $G_0':=\pi[G_0\cap M]$ is $\dM(\nu)$-generic over $V$. Let $H_0'$ be the $\pi(\dot{\dL})$-generic filter induced by $G$. Then $\pi^{-1}[H_0']\cup\{\sigma^G\}$ is (by elementarity) a ${<}\,\kappa$-sized collection of pairwise compatible conditions in $\dot{\dL}^G$. Ergo there exists a lower bound $r$. If now $H_0$ is $\dot{\dL}^G$-generic containing $r$, $\pi[H_0\cap M]=H_0'$ is $\pi(\dot{\dL})$-generic over $V[G_0']$. This shows that for any model $M$ as above, there exists a condition $q_M\leq (p,\sigma)$ forcing that for any $\dM(\kappa)*\dot{\dL}$-generic filter $G*H$, $\pi[G*H\cap M]$ is equal to the $\pi(\dM(\kappa)*\dot{\dL})$-generic filter induced by $G$.
		
		Now we transform, as before, $\dot{f}$ into a ground-model $(\kappa,\lambda)$-list. Let $a\in[\lambda]^{<\kappa}$.
		\begin{itemize}
			\item If there exists $M$ with the conditions listed above such that $a=M\cap\lambda$ and for some $q_a\leq q_M$ and a $\pi(\dM*\dot{\dL})$-name $\dot{x}_a$, $q_a\Vdash\dot{f}(\check{a})=\pi^{-1}[\dot{x}_a^{\pi[\Gamma\cap M]}]$, let
			$$g(a):=\{\langle\alpha,\beta\rangle\;|\;\alpha,\beta\in\pi[a]\wedge\alpha\Vdash\check{\beta}\in\dot{x}_a\}$$
			and
			$$e(a):=\pi^{-1}[g(a)]$$
			\item Otherwise, let $e(a):=\emptyset$.
		\end{itemize}
		
		By Lemma \ref{LaverIneffable}, there exists a $\lambda$-ineffability witness $M$ for $\kappa$ with respect to $e$ such that
		\begin{enumerate}
			\item $\nu:=M\cap\kappa\in\kappa$ is inaccessible
			\item $[M\cap\lambda]^{<\nu}\subseteq M$
			\item $\{\kappa,\lambda,\dM(\kappa),\dot{\dL},\dot{F},l\}\cup p\subseteq M$
			\item $l(\nu)=\pi(\dot{\dL})$.
		\end{enumerate}
		Denote $a:=M\cap\lambda$ and let $G_1*H_1$ be $\dM(\kappa)*\dot{\dL}$-generic containing $p_M$. We will use $G_1*H_1$ to show the existence of $p_a$. Because $G_1*H_1$ contains $p_M$, $G_1'*H_1':=\pi[G_1*H_1\cap M]$ is $\pi(\dM(\kappa)*\dot{\dL})$-generic over $V$. In particular, it is $\pi(\dM(\kappa)*\dot{\dL})$-generic over $N$ and the following holds:
		\begin{enumerate}
			\item $M[G_1*H_1]\cap V=M$,
			\item $\pi_{M[G_1*H_1]}\colon M[G_1*H_1]\to N[G_1'*H_1']$,
			\item $\pi_{M[G_1*H_1]}\uhr M=\pi$.
		\end{enumerate}
		For simplicity we write $\pi_{M[G_1*H_1]}=\pi$.
		
		Toward a contradiction, assume that $\pi[\dot{f}(\check{a})]\notin V[G_1'*H_1']$. As $V[G_1*H_1]$ is an extension of $V[G_1'*H_1']$ by a forcing with the ${<}\,\tau^+$-approximation property, there exists $z\in[\pi(\lambda)]^{<\tau^+}\cap V[G_1'*H_1']$ such that $z\cap\pi[\dot{f}^{G_1*H_1}(a)]\notin V[G_1'*H_1']$. Because $H_1'$ is generic for a ${<}\,\nu$-directed closed forcing, $z\in V[G_1']$ and by Lemma \ref{ClosureAfterForcing}, $z\in N[G_1']$, so $z=\pi(y)$ for $y\in M[G_1]\subseteq M[G_1*H_1]$. Furthermore, $M[G_1*H_1]\cap H(\Theta')$ is closed under $\dot{F}^{G_1*H_1}$, so it witnesses the ${<}\,\tau^+$-slenderness of $\dot{f}^{G_1*H_1}$. But then
		$$\pi[\dot{f}^{G_1*H_1}(a)]\cap z=\pi[\dot{f}^{G_1*H_1}(a)]\cap \pi(y)=\pi[\dot{f}^{G_1*H_1}(a)\cap y]=\pi(\dot{f}^{G_1*H_1}(a)\cap y)\in N[G_1'*H_1']$$
		a contradiction. So $\pi[\dot{f}^{G_1*H_1}(a)]\in V[G_1'*H_1']$ and there exists a condition $q_a$ as required.
		
		Lastly, we want to show that $q_a$ forces that $M[\Gamma]$ is a $\lambda$-ineffability witness for $\kappa$ with respect to $\dot{f}$. To this end, let $G_2*H_2$ be $\dM(\kappa)*\dot{\dL}$-generic containing $q_a$. By assumption we are in case (1) and there is $b_e\in M$ such that $b_e\cap M=e(M\cap\lambda)$. Define
		$$b_f:=\{\beta\;|\;\exists\alpha\in G_2*H_2\;\langle\alpha,\beta\rangle\in b_e\}\in M[G_2*H_2]$$
		we are left to show
		$$b_f\cap M[G_2*H_2]=b_f\cap M=\dot{f}^{G_2*H_2}(M\cap\lambda)=\dot{f}^{G_2*H_2}(M[G_2*H_2]\cap\lambda)$$
		where the last equality holds because $q_a\leq q_M$, which implies $M[G_2*H_2]\cap V=M$.
		
		Let $\beta\in b_f\cap M$. By elementarity there is $\alpha\in G_2*H_2\cap M$ such that $\langle \alpha,\beta\rangle\in b_e$ and we have $\langle\alpha,\beta\rangle\in b_e\cap M=e(a)=\pi^{-1}(g(a))$. So $\pi(\langle\alpha,\beta\rangle)=\langle\pi(\alpha),\pi(\beta)\rangle\in g(a)$. By the definition, $\pi(\alpha)\Vdash\pi(\check{\beta})\in\dot{x}_a$. Hence $\pi(\beta)\in\dot{x}_a^{\pi[G_2*H_2\cap M]}$ and thus $\beta\in\dot{f}^{G_2*H_2}(a)$.
		
		Let $\beta\in\dot{f}^{G_2*H_2}(a)$. Then $\pi(\beta)\in\dot{x}_a^{\pi[G_2*H_2\cap M]}$. Ergo there exists $\pi(\alpha)\in\pi[G_2*H_2\cap M]$ such that $\pi(\alpha)\Vdash\pi(\beta)\in\dot{x}_a$. Hence $\pi(\langle\alpha,\beta\rangle)\in g(a)$ and $\langle\alpha,\beta\rangle\in e(a)=b_e\cap M$. Thus $\beta\in b_f\cap M$.
	\end{proof}
	
	\section{Open Questions}
	
	We finish by stating two well-known questions which might have become more tractable with the new techniques used in this paper.
	
	\begin{myque}
		Is it consistent to have $\ISP(\omega_1,\omega_{n+2},\lambda_n)$ for every $n\in\omega$?
	\end{myque}

	A model of the former would most likely require $2^{\omega}=\aleph_{\omega+1}$. With a ``better`` powerset sequence, one could hope to find a model of the following:
	
	\begin{myque}
		Is it consistent to have $\ISP(\omega_{n+2},\omega_{n+2},\lambda_n)$ for every $n\in\omega$ together with $2^{\omega_n}=\omega_{n+2}$?
	\end{myque}

	We have shown a slight variation of the iterands of the Cummings-Foreman iteration (see \cite{CummingsForemanTreeProperty}) can force $\ISP(\tau^+,\kappa,\geq\kappa)$ to be indestructible under ${<}\,\kappa$-directed closed forcing (which is an important step in showing that the tree property holds at all $\aleph_{n+2}$), hopefully making progress toward an answer to the questions.

	\printbibliography

\end{document}